\documentclass[11pt,reqno]{article}
\usepackage[english]{babel}
\usepackage{amsmath}
\usepackage{amssymb, amsthm}
\usepackage{graphicx}
\usepackage{color}
\usepackage{bm}
\usepackage{mathtools}
\usepackage{subfigure}
\usepackage{stmaryrd}

\newcommand{\nm}{\noalign{\smallskip}}
\newcommand{\ds}{\displaystyle}
\usepackage{pgf}
\usepackage{pgffor}
\usepgflibrary{plothandlers}
\usepackage{pgfplots}
\pgfplotsset{compat=newest}
 \pgfplotsset{width=15cm}
\pgfplotsset{plot coordinates/math parser=false}
\newlength\figureheight
\newlength\figurewidth

\newtheorem{de}{Definition}[section]
\newtheorem{theo}{Theorem}[section]

\newtheorem{prop}[theo]{Proposition}
\newtheorem{lemm}[theo]{Lemma}

\newcommand{\C}{{\cal C}}

\newcommand{\R}{\mathbb{R}}

\numberwithin{equation}{section} \numberwithin{figure}{section}

\newcommand{\dr}{\partial}

\newcommand{\g}{\nabla}

\newcommand{\Bbf}{\mathbf{B}}
\newcommand{\Ebf}{\mathbf{E}}

\newcommand{\Jbf}{\mathbf{J}}
\newcommand{\Abf}{\mathbf{A}}
\newcommand{\Fbf}{\mathbf{F}}
\newcommand{\fbf}{\mathbf{f}}

\newcommand{\ubf}{\mathbf{u}}

\newcommand{\ebf}{\mathbf{e}}

\newcommand{\nubf}{\boldsymbol{\nu}}

\title{A mathematical and numerical framework for magnetoacoustic tomography with magnetic induction\thanks{\footnotesize This work was supported  by the
ERC Advanced
Grant Project MULTIMOD--267184.}}
\author{Habib Ammari\thanks{\footnotesize Department of Mathematics and Applications,
Ecole Normale Sup\'erieure, 45 Rue d'Ulm, 75005 Paris, France
(habib.ammari@ens.fr, simon.boulmier@polytechnique.edu, pierre.millien@ens.fr).} \and 
\and Simon Boulmier\footnotemark[2] \and Pierre Millien\footnotemark[2] }

\begin{document}
\maketitle
\begin{abstract}
We provide a mathematical analysis and a numerical framework for magnetoacoustic tomography with magnetic induction. The imaging problem is to reconstruct the conductivity distribution of biological tissue from measurements of the Lorentz force induced tissue vibration.  We begin with reconstructing from the acoustic measurements the divergence of the Lorentz force, which is acting as the source term in the acoustic wave equation. Then we recover the electric current density from the divergence of the Lorentz force. To solve the nonlinear inverse conductivity problem, we introduce an optimal control method for reconstructing the conductivity from the electric current density.
We prove its convergence and stability. We also present a point fixed approach and prove its convergence to the true solution. A new direct reconstruction scheme involving a partial differential equation is then proposed based on viscosity-type regularization to a transport equation satisfied by the electric current density field. We prove that solving such an equation yields the true conductivity distribution as the regularization parameter approaches zero. Finally, we test the three schemes numerically in the presence of measurement noise, quantify their stability and resolution, and compare their performance.
\end{abstract}

\bigskip

\noindent {\footnotesize Mathematics Subject Classification
(MSC2000): 35R30, 35B30.}


\section{Introduction}

The Lorentz force plays a key role in magneto-acoustic tomographic techniques \cite{rothrev}. Several approaches have been developed with the aim of providing electrical impedance information at a spatial resolution on the scale of ultrasound wavelengths. These include  ultrasonically-induced Lorentz force imaging \cite{ammari2014mathematical, GraslandMongrain2013} and magneto-acoustic tomography with magnetic induction \cite{xu2005magnetoacoustic,roth}.  

Electrical conductivity varies widely among soft tissue types and pathological states \cite{Foster1989, morimoto1993study} and its measurement can provide
information about the physiological and pathological conditions of tissue \cite{laure}.
Acousto-magnetic tomographic techniques have the potential to detect small conductivity inhomogeneities, enabling them to diagnose pathologies such as cancer by detecting tumorous tissues when other conductivity imaging techniques fail to do so.

In magnetoacoustic imaging with magnetic induction, magnetic fields are used to induce currents in the tissue. Ultrasound is generated by placing the tissue in a dynamic and static magnetic field. The dynamic field induces eddy currents and the static field leads to generation of acoustic vibration from Lorentz force on the induced currents.
The divergence of the Lorentz force acts as acoustic source of  propagating ultrasound waves that can be sensed by ultrasonic transducers placed around the tissue.   
The imaging problem is to obtain the conductivity distribution of the tissue from the acoustic source map; see \cite{bin1,bin2,bin3,bin4,bin5}. 

This paper provides a  mathematical and numerical framework for magnetoacoustic imaging with magnetic induction. We develop efficient methods for reconstructing the conductivity in the medium from the Lorentz force induced vibration. For doing so, we first estimate the electric current density in the tissue. Then we design efficient algorithms for reconstructing the heterogeneous conductivity map from the electric current density with the ultrasonic resolution.  

The paper is organized as follows. We start by describing the forward  problem.
Then we reconstruct from the acoustic measurements the divergence of the Lorentz force, which is acting as the source term in the acoustic wave equation. We recover the electric current density from the divergence of the Lorentz force, which  reduces the problem to imaging the conductivity from the internal electric current density.
We introduce three reconstruction schemes for solving the conductivity imaging problem from the internal electric current density. The first is an optimal control method. 
One of the contributions of this paper is the proof of convergence and stability of the optimal control approach provided that two magnetic excitations leading to nonparallel current densities are employed. Then we present a point fixed approach and prove that it converges to the true conductivity image. 
Finally, we propose an alternative to these iterative schemes via the use of a transport equation satisfied by the internal electric current density. Our third algorithm is direct and can be viewed as a PDE-based reconstruction scheme. 
We test numerically the three proposed schemes in the presence of measurement noise, and also quantify their stability and resolution.

The feasibility of imaging of Lorentz-force-induced motion in conductive samples was shown in \cite{mri}. The magnetoacoustic tomography with magnetic induction investigated here was 
experimentally tested in \cite{bin3,bin4}, and was
reported to produce conductivity images of quality comparable to that of ultrasound images taken under similar conditions. Other emerging hybrid techniques for conductivity imaging have also been reported in \cite{stanford, AMMARI-08, AMMARI-BONNETIER-CAPDEBOSCQ-08, kozh, handbook, GEBAUER-SCHERZER-08, sirev, seobook, otmar2}.

\section{Forward problem description}

\subsection{Time scales involved}
The forward problem in magnetoacoustic tomography with magnetic induction (MAT-MI) is  multiscale in nature. The different phenomena involved in the experiment evolve on very different time scales. Precisely, there are three typical times that appear in the mathematical model for MAT-MI.
\begin{itemize}
\item The first one is the time needed for an electromagnetic wave to propagate in the medium and is denoted by $\tau_{\text{em}}$. Typically, if the medium has a diameter of $1cm$, we have $\tau_{\text{em}}\sim 10^{-11}s$.
\item  The second characteristic time length, denoted by $\tau_{\text{pulse}}$ of the experiment is the time width of the magnetic pulse sent into the medium. Since, the time-varying magnetic field is generated by discharging a capacitor, $\tau_{\rm pulse}$ is in fact the time needed to discharge the capacitor such that  $\tau_{\text{pulse}} \sim 1\mu s$ \cite{xu2005magnetoacoustic}. 
\item The third characteristic time, $\tau_{\text{sound}}$, is the time consumed by the acoustic wave to propagate through the medium. The speed of sound is about $1.5 \cdot 10^3 m.s^{-1}$ so $\tau_{\text{sound}} \sim 6 \mu s$ for a medium of $1cm$ diameter.
\end{itemize}

\subsection{Electromagnetic model}
Let $(\ebf_i)_{i=1,2,3}$ be an orthonormal basis of $\R^3$.
Let $\Omega$ be a three-dimensional bounded $\mathcal{C}^1$ convex domain. The medium is assumed to be non magnetic, and its conductivity is given by $\sigma$ (the question of the regularity of $\sigma$ will arise later).
Assume that the medium $\Omega$ is placed in  a uniform, static magnetic field in the transverse direction $\Bbf_0=B_0 \ebf_3$.

\subsubsection{The magnetoquasistatic regime}
At time $t=0$ a second time varying magnetic field is applied in the medium. The time varying magnetic field has the form $\Bbf_1(x,t)=B_1(x)u(t)\ebf_3$. $B_1$ is assumed to be a known smooth function and $u$ is the shape of the stimulating pulse. The typical width of the pulse is about $1\mu s$ so we are in presence of a slowly varying magnetic-field. 
This regime can be described by the magnetoquasistatic equations \cite{larsson2007electromagnetics}, where the propagation of the electrical currents is considered as instantaneous, but, the induction effects are not neglected. 
These governing equations in $\Omega \times \mathbb{R}_+$ are 
\begin{equation}
\g \cdot \Bbf = 0,
\end{equation}
\begin{equation}
\g \times \Ebf = -\frac{\dr \Bbf}{\dr t},
\end{equation}
and
\begin{equation}\label{eq:divJ}
\g \cdot \Jbf = 0,
\end{equation}
where $\varepsilon_0$ is the permittivity of free space, $\Bbf$ is the total magnetic field in the medium and $\Ebf$ is the total electric field in the medium.
Ohm's law is valid and is expressed as 
\begin{equation}
\Jbf = \sigma \Ebf  \quad \mbox{in } \Omega \times \mathbb{R}_+,
\end{equation}
where $\sigma$ is the electrical conductivity of the medium. For now on, we assume that $\sigma \in L^\infty_{a,b}(\Omega) $, where 
$$L^\infty_{a,b}(\Omega):=\{f\in L^\infty(\Omega^\prime): a<f<b \mbox { in } \Omega^\prime, \quad f \equiv \sigma_0 \mbox{ in } \Omega \setminus \overline{\Omega^\prime}\}$$ with $\sigma_0, a,$ and $b$ being three given positive constants, $0<a<b$, and $\Omega^\prime \Subset \Omega$.

 As in \cite{larsson2007electromagnetics}, we use the Coulomb gauge ($\g\cdot \Abf =0$)  to express the potential representation of the fields $\Bbf$ and $\Ebf$.
The magnetic field $\Bbf$ is written as
\begin{equation} 
\Bbf = \g \times \Abf,
\end{equation} and the electric field $\Ebf$ is then of the form
\begin{equation}\label{eq:Epotential}
\Ebf = -\g \widetilde{V} - \frac{\dr \Abf}{\dr t} \quad \mbox{in } \Omega \times \mathbb{R}_+,
\end{equation}
where $\widetilde{V}$ is the electric potential. Writing $\Abf$ as follows:
$$
\Abf(x,t)= \Abf_0(x) + \Abf_1(x) u(t),$$
where $\Abf_0$ and $\Abf_1$ are assumed to be smooth.  In view of (\ref{eq:divJ}) and (\ref{eq:Epotential}), we look for $\widetilde{V}(x,t)$ of the form  $\widetilde{V}(x,t)= V(x) u^\prime(t)$ with $V$ satisfying
\begin{equation*}
\g\cdot \sigma\g V = - \g\cdot \sigma {\Abf}_1 \quad \mbox{in } \Omega \times \mathbb{R}_+.
\end{equation*}
The boundary condition on $V$ can be set as a Neumann boundary condition. Since the medium $\Omega$ is usually embedded in a non-conductive medium (air), no currents leave the medium, i.e., $\Jbf \cdot \nubf =0$ on $\partial \Omega$, where $\nubf$ is the outward normal at $\partial \Omega$. To make sure that the boundary-value problem satisfied by $V$ is well posed, we add the condition $\int_{\Omega} V=0$. We have the following boundary value problem for $V$:
\begin{equation}\label{eq:V}
\left\{\begin{aligned}
\g\cdot \sigma\g V =& - \g\cdot\sigma {\Abf}_1 \quad &\text{in } \Omega,\\
\sigma \frac{\dr V}{\dr \nu} = & - \sigma  \Abf_1\cdot \nubf \quad & \text{on } \dr \Omega,\\
\int_{\Omega} V=& 0.
\end{aligned}\right.
\end{equation}

%


\subsection{The acoustic problem}
\subsubsection{Elasticity formulation}

The eddy currents induced in the medium, combined with the magnetic field, create a Lorentz force based stress in the medium. The Lorentz force $\fbf$ is determined as 
\begin{equation}
\fbf= \Jbf \times \Bbf  \quad \mbox{in } \Omega \times \mathbb{R}_+.
\end{equation}

Since the duration and the amplitude of the stimulation are both small, we assume that we can use the linear elasticity model. The displacements inside the medium can be described by the initial boundary-value problem for the Lam\'e system of equations
\begin{equation}\label{eq:elasto}
\left\{\begin{aligned}
\rho \dr_t^2 \ubf -\g \lambda \g\cdot \ubf - \g\cdot \mu \g^s \ubf = \Jbf \times \Bbf \quad \text{in }&  \Omega \times \mathbb{R}_+, \\
\frac{\partial \ubf}{\partial n}  =0 \quad \text{on } &
 \dr \Omega \times \mathbb{R}_+,\\
 \ubf(x,0) = \frac{\partial \ubf}{\partial t}(x,0)=0 \quad \text{in } &
 \Omega,
\end{aligned}\right.
\end{equation}
where $\lambda$ and $\mu$ are the Lam\'e coefficients, 
$\rho$ is the density of the medium at rest, and $\nabla^s \ubf = (\nabla \ubf + \nabla^T \ubf)/2$ with 
the superscript $T$ being the transpose.
Here, $\partial /\partial n$ denotes the co-normal derivative defined by
$$
\frac{\partial \ubf}{\partial n} = \lambda (\nabla \cdot \ubf) \nubf + 
2 \mu  \g^s \ubf \, \nubf \quad \mbox{on } \partial \Omega, 
$$
where $\nubf$ is the outward normal at $\partial \Omega$.  The functions $\lambda$, $\mu$, and $\rho$ are assumed to be positive, smooth functions on $\overline{\Omega}$. 

The Neumann boundary condition, ${\partial \ubf}/{\partial n} = 0 $ on $\partial \Omega$,  comes from the fact that the sample is embedded in air  and can move freely at the boundary.

\subsubsection{The acoustic wave}

Under some physical assumptions, the Lam\'e system of equations (\ref{eq:elasto}) can be reduced to an acoustic wave equation. For doing so, 
we neglect the shear effects in the medium by taking $\mu =0$. The acoustic  approximation says that the dominant wave type is a compressional wave.  Equation (\ref{eq:elasto}) becomes 
 \begin{equation}\label{eq:elastoconstant}
\left\{\begin{aligned}
\rho \dr_t^2 \ubf -\g \lambda \g\cdot \ubf = \Jbf \times \Bbf \quad \text{in }&  \Omega \times \mathbb{R}_+, \\
\frac{\partial \ubf}{\partial n} = 0 \quad \text{on } & \dr \Omega \times \mathbb{R}_+, \\
\ubf(x,0) = \frac{\partial \ubf}{\partial t}(x,0)=0 \quad \text{in } &
 \Omega.
\end{aligned}\right.
\end{equation}
Introduce the pressure
$$
p= \lambda \nabla \cdot \ubf \quad \mbox{ in } \Omega \times \mathbb{R}_+.
$$
Taking the divergence of (\ref{eq:elastoconstant}) yields the acoustic wave equation
 \begin{equation}\label{eq:acoustic}
\left\{\begin{aligned}
\frac{1}{\lambda} \frac{\partial^2 p}{\partial t^2} - \nabla \cdot \frac{1}{\rho} \nabla  p = 
\g \cdot \frac{1}{\rho} \left(\Jbf \times \Bbf\right) \quad \text{in }&  \Omega \times \mathbb{R}_+,\\
 p =  0 \quad \text{on } & \dr \Omega \times \mathbb{R}_+,
 \\
 p(x,0)= \frac{\partial p}{\partial t}(x,0) = 0  \quad \text{in }&  \Omega .
\end{aligned}\right.
\end{equation}

We assume that the duration $T_{\text{pulse}}$ of the electrical pulse sent into the medium is short enough so that $p$ is the solution to
\begin{equation}\label{eq:acousticfinal}
\left\{\begin{aligned}
\frac{1}{\lambda}\frac{\partial^2 p}{\partial t^2} (x,t) - \nabla \cdot \frac{1}{\rho} \nabla  p (x,t) = 
f(x) \delta_{t=0} \quad \text{in }&  \Omega \times \mathbb{R}_+,\\
 p =  0 \quad \text{on } & \dr \Omega \times \mathbb{R}_+,
 \\
 p(x,0)= \frac{\partial p}{\partial t}(x,0) = 0  \quad \text{in }&  \Omega , 
\end{aligned}\right.
\end{equation}
where 
\begin{equation}
\label{deff}
\ds f(x)=\int_0^{T_{\text{pulse}}} \g\cdot(  \frac{1}{\rho}  \Jbf(x,t) \times \Bbf(x,t)) dt.
\end{equation}

Note that acoustic wave reflection in soft tissue by an interface
with air can be modeled well by a homogeneous Dirichlet boundary condition; see, for instance, \cite{wang}.

Let \begin{equation*}
g(x,t)= \frac{\partial p}{\partial \nu} (x,t), \quad \forall (x,t) \in  \dr \Omega \times \mathbb{R}_+.
\end{equation*}

In the next section, we aim at reconstructing the source term $f$ from the data $g$.

\section{Reconstruction of the acoustic source}
%
%

%
%
%

In this subsection, we assume that $\lambda = \lambda_0 + \delta \lambda $ and $\rho=\rho_0 + \delta \rho$, where the functions $\delta \lambda$ and $\delta \rho$ are such that $||\delta \lambda||_{L^\infty(\Omega)} \ll \lambda_0$ and 
$||\delta \rho||_{L^\infty(\Omega)} \ll \rho_0$.  We assume that $\lambda, \lambda_0, \rho,$ and $\rho_0$ are known 
and denote by $c_0= \sqrt{\frac{\lambda_0}{\rho_0}}$ the background acoustic speed. Based on the Born approximation, we image the source term $f$. 
To do so, we first consider the time-harmonic regime and define $\Gamma^\omega$ to be the outgoing fundamental solution to $\Delta + \frac{\omega^2}{c_0^2}$:
\begin{equation}\label{eq:defG0}
\left(\Delta_x + \frac{\omega^2}{c_0^2} \right) \Gamma^\omega(x,y)= \delta_y(x),
\end{equation}
subject to the Sommerfeld radiation condition:
\begin{equation*}
\vert x \vert^{\frac{1}{2}} \left( \frac{\dr}{\dr \vert x\vert} \Gamma^\omega(x,y)- i \frac{\omega}{c_0} \Gamma^\omega(x,y)\right) \longrightarrow 0, \quad \vert x \vert \rightarrow \infty.
\end{equation*}
We need the following integral operator $(\mathcal{K}_\Omega^{\omega})^\star: L^2(\partial \Omega) \rightarrow L^2(\partial \Omega)$ given by
$$
(\mathcal{K}_\Omega^{\omega})^\star[\phi](x)= \int_{\partial \Omega}
\frac{\partial \Gamma^\omega}{\partial \nu(x)}(x,y) \phi(y)\, ds(y). 
$$
Let $G_\Omega^\omega$ be the Dirichlet Green function for $\Delta +  \frac{\omega^2}{c_0^2}$ in $\Omega$, i.e., for each $y\in \Omega$, $G_\Omega^\omega(x,y)$ is the solution to 
$$
\left\{\begin{array}{l}
\left(\Delta_x + \frac{\omega^2}{c_0^2}\right) G_\Omega^\omega(x,y)= \delta_y(x), \quad 
x\in \Omega,\\
G_\Omega^\omega(x,y) = 0, \quad x \in \partial \Omega.
\end{array}
\right. $$

Let $\hat{p}$ denote the Fourier transform of the pressure $p$ and $\hat{g}$ the Fourier transform of $g$. The function $\hat{p}$ is the solution to the Helmholtz equation:
$$
\left\{\begin{aligned}
\frac{\omega^2}{\lambda(x)} \hat{p}(x,\omega) + \nabla \cdot \frac{1}{\rho(x)} \nabla \hat{p}(x,\omega) = {f}(x), \quad & x \in \Omega, \\
 \hat{p}(x,\omega) = 0, \quad & x \in \partial \Omega. 
\end{aligned}\right.
$$
Note that $f$ is a real-valued function. 

The Lippmann-Schwinger representation formula shows that
$$\begin{array}{lll}
\hat{p}(x,\omega) &=& \ds \int_\Omega (\frac{\rho_0}{\rho(y)} -1) \nabla \hat{p}(y, \omega) \cdot \nabla G_\Omega^\omega(x,y) \, dy - \omega^2 \int_\Omega (\frac{\rho_0}{\lambda(y)} - \frac{\rho_0}{\lambda_0}) \hat{p}(y, \omega) G_\Omega^\omega(x,y) \, dy \\ \nm && \ds + \rho_0 \int_\Omega {f}(y) G_\Omega^\omega(x,y) \, dy.
\end{array}$$
Using the Born approximation, we obtain
$$\begin{array}{lll}
\hat{p}(x,\omega) &\approx&\ds - \frac{1}{\rho_0} \int_\Omega \delta \rho(y)  \nabla \hat{p}_0(y, \omega) \cdot \nabla G_\Omega^\omega(x,y) \, dy + \frac{\omega^2}{c_0^2} 
\int_\Omega \frac{\delta \lambda(y)}{\lambda_0} \hat{p}_0(y, \omega) G_\Omega^\omega(x,y) \, dy \\ \nm && \ds + \rho_0 \int_\Omega {f}(y) G_\Omega^\omega(x,y) \, dy
\end{array}
$$
for $x \in \Omega$, where
$$
\hat{p}_0(x,\omega) := \rho_0 \int_\Omega {f}(y) G_\Omega^\omega(x,y) \, dy, \quad x \in \Omega. 
$$
Therefore, from the identity \cite[Eq. (11.20)]{ammarikang}
$$
(\frac{1}{2} I + (\mathcal{K}_\Omega^{\omega})^\star) \bigg[\frac{\partial G^\omega_\Omega}{\partial \nu_{\cdot}}(\cdot, y)\bigg](x) = \frac{\partial \Gamma^\omega}{\partial \nu(x)}(x, y), 
\quad x \in \partial \Omega, y \in \Omega,$$
it follows that
$$\begin{array}{lll}
\ds (\frac{1}{2} I + (\mathcal{K}_\Omega^{\omega})^\star) [\hat{g}](x,\omega)  &\approx&\ds - \frac{1}{\rho_0} \int_\Omega \delta \rho(y)  \nabla \hat{p}_0(y, \omega) \cdot \nabla \frac{\partial \Gamma^\omega(x,y)}{\partial \nu(x)} \, dy + \frac{\omega^2}{c_0^2} 
\int_\Omega \frac{\delta \lambda(y)}{\lambda_0} \hat{p}_0(y, \omega) \frac{\partial \Gamma^\omega(x,y)}{\partial \nu(x)} \, dy \\ \nm && \ds + \rho_0 \int_\Omega {f}(y) \frac{\partial \Gamma^\omega(x,y)}{\partial \nu(x)} \, dy
\end{array}
$$
for $x \in \partial \Omega$. 

Introduce 
$$
I(z, \omega) := \int_{\partial \Omega}  \bigg[
\overline{\Gamma^\omega}(x,z) 
\ds (\frac{1}{2} I + (\mathcal{K}_\Omega^{\omega})^\star) [\hat{g}](x,\omega) \, -  {\Gamma^\omega}(x,z) 
\ds \overline{(\frac{1}{2} I + (\mathcal{K}_\Omega^{\omega})^\star) [\hat{g}]}(x,\omega) \bigg] \, ds(x)
$$
for $z\in \Omega$. 

We recall the Helmholtz-Kirchhoff identity \cite[Lemma 2.32]{ammari2013mathematical}
$$
\int_{\partial \Omega}  \bigg[
\overline{\Gamma^\omega}(x,z)  \frac{\partial {\Gamma^\omega}(x,y)}{\partial \nu(x) }
-  {\Gamma^\omega}(x,z)  \overline{\frac{\partial {\Gamma^\omega}(x,y)}{\partial \nu(x) }} \bigg] \, ds(x)  = 2 i \Im m \Gamma^\omega(z,y). 
$$
We also recall that ${f}$ is real-valued and write 
${f} \approx {f}^{(0)} + \delta {f}$. Given $I(z, \omega)$ we solve the deconvolution problem 
\begin{equation} \label{decon1}
2 i \rho_0 \int_\Omega \Im m \Gamma^\omega (z,y) {f}^{(0)} (y) \, dy = I(z, \omega), 
\quad z \in \Omega, 
\end{equation}
 in order to reconstruct $f^{(0)}$ with a resolution limit determined by the Rayleigh criteria. Once $f^{(0)}$ is determined, we  solve the second deconvolution problem (\ref{decon2}) 
 \begin{equation} \label{decon2}
2 i \rho_0 \int_\Omega \Im m \Gamma^\omega (z,y) \delta {f}(y) \, dy = \delta I(z, \omega), \quad z \in \Omega, 
\end{equation}
 to find the correction $\delta f$. 
Here, 
$$\delta I(z, \omega) := \int_{\partial \Omega}  \bigg[
\overline{\Gamma^\omega}(x,z) 
\ds \delta \hat{g} (x,\omega) \, -  {\Gamma^\omega}(x,z) 
\ds \overline{\delta \hat{g}}(x,\omega) \bigg] \, ds(x)
$$
with
$$
\delta \hat{g}(x, \omega)= \frac{1}{\rho_0} \int_\Omega \delta \rho(y)  \nabla \hat{p}^{(0)}(y, \omega) \cdot \nabla \frac{\partial \Gamma^\omega(x,y)}{\partial \nu(x)} \, dy + \frac{\omega^2}{c_0^2} 
\int_\Omega \frac{\delta \lambda(y)}{\lambda_0} \hat{p}^{(0)}(y, \omega) \frac{\partial \Gamma^\omega(x,y)}{\partial \nu(x)} \, dy ,
$$
and
$$
\hat{p}^{(0)}(x,\omega) := \rho_0 \int_\Omega {f}^{(0)}(y) G_\Omega^\omega(x,y) \, dy, \quad x \in \Omega. 
$$
Since  by Fourier transform, $\hat{g}$ is known for all  $\omega \in \mathbb{R}_+$, $I(z, \omega)$ can be computed for all  $\omega \in \mathbb{R}_+$. Then from the identity
\cite[Eq. (1.35)]{ammari2013mathematical}
$$
\frac{2}{\pi} \int_{\mathbb{R}_+} \omega \Im m \Gamma^\omega(x,z) \, d\omega = - \delta_z(x),
$$
where $\delta_z$ is the Dirac mass at $z$, it follows that
$$
f^{(0)}(z) = \frac{1}{i \pi \rho_0} \int_{\mathbb{R}_+} \omega I(z,\omega)\, d\omega
\quad \mbox{and } \delta f(z) = \frac{1}{i \pi \rho_0} \int_{\mathbb{R}_+} \omega \delta I(z,\omega)\, d\omega.
$$
We refer to \cite{gang, gang2} and the references therein for source reconstruction approaches with finite set of frequencies.

\section{Reconstruction of the conductivity}
We assume that we have reconstructed the pressure source $f$ given by (\ref{deff}).
We also assume that the sample $\Omega$ is thin and hence can be assimilated to a two dimensional domain. Further, we suppose that $\Omega \subset \text{vect }(\ebf_1,\ebf_2)$. Here,  $\text{vect }(\ebf_1,\ebf_2)$ denotes the vector space spanned by $\ebf_1$  and $\ebf_2$. Recall that the magnetic fields $\Bbf_0$ and $\Bbf_1$ are parallel to 
$\ebf_3$. 
We write $\Jbf(x,t) = \Jbf(x) u^\prime(t)$. 
 In order to recover the conductivity distribution, we  start by reconstructing the vector field $\Jbf(x)$ in $\Omega$. 

\subsection{Reconstruction of the electric current density}

\subsubsection{Helmholtz decomposition}
Let $H^1(\Omega):= \{ v \in L^2(\Omega): \nabla v \in L^2(\Omega)\}$. Let $H_0^1(\Omega)$ be the set of functions in $H^1(\Omega)$ with trace zero on $\partial \Omega$ and let $H^{-1}(\Omega)$ be the dual of $H_0^1(\Omega)$. 

 We need the following two classical results.


\begin{lemm}\label{lem:regularityJ}
If $\sigma \in L_{a,b}^\infty(\Omega)$ then the solution $V$ of  (\ref{eq:V}) belongs to $H^1(\Omega)$ and hence, the electric current density $\Jbf$ belongs to $L^2(\Omega)$.
\end{lemm}
The following Helmholtz decomposition in two dimensions holds \cite{sohr2012navier}. 
\begin{lemm}\label{lem:helmholtz2d}
If $\mathbf{f}$ is a vector field in ${L}^2(\Omega)$, then there exist two functions  $v \in H^1(\Omega)$ and ${w} \in  H^1(\Omega)$ such that
\begin{equation} \label{helmdecomp}
\mathbf{f} =\g v+ \textbf{curl } w.
\end{equation}
The differential operator  \textbf{curl} is defined by $\textbf{curl } w = (-\dr_2 w, \dr_1 w)$. Furthermore, if $\nabla \cdot \mathbf{f} \in {L}^2(\Omega)$, then the potential $v$ is a solution to 
\begin{equation} \label{eq:eqv}
\left\{\begin{aligned}
-\Delta v = \g \cdot \mathbf{f} \quad &\text{ in } \Omega,\\
\frac{\dr v}{\dr \nu} = \mathbf{f}\cdot \nubf \quad &\text{ on } \dr \Omega,
\end{aligned}\right.
\end{equation}
and $w$ is the unique solution of
\begin{equation}\label{eq:eqw}
\int_\Omega \textbf{curl } w \cdot \textbf{curl } \phi  = \int_\Omega (\mathbf{f} - \g v) \cdot \textbf{curl } \phi, \quad \forall \phi \in H(\Omega),
\end{equation} 
where $H(\Omega)=\{ \phi \in L^2(\Omega), \nabla \times \phi \in L^2(\Omega), \nabla \cdot \phi =0\}$. 
The problem can be written in strong form in $H^{-1}(\Omega)$: 
\begin{equation*}
\left\{ \begin{aligned}
-\Delta w = \text{curl }\mathbf{f}\quad &\text{ in } \Omega,\\
w = 0 \quad &\text{ on } \dr \Omega,
\end{aligned}\right. 
\end{equation*}
where the operator $\text{curl}$ is defined on vector fields by $\text{curl } \mathbf{f}
 = -\dr_2 f_1 + \dr_1 f_2.$
\end{lemm}
We apply the Helmholtz decomposition (\ref{helmdecomp}) to the vector field $\Jbf \in {L}^2(\Omega)$ and get the following proposition.
\begin{prop}\label{prop:Jdecomposition}
 There exists a function $w \in H$ such that
\begin{equation}\label{eq:decompJ}
\Jbf= \textbf{curl } w ,
\end{equation} and $w$ is the unique solution of 
\begin{equation}\label{eq:eqwstrong}
\left\{ \begin{aligned}
-\Delta w = \text{curl }\Jbf\quad &\text{ in } \Omega,\\
w = 0 \quad &\text{ on } \dr \Omega.
\end{aligned}\right.
\end{equation}
\end{prop}
Recall (\ref{eq:divJ}):
\begin{equation*}
\g \cdot  \Jbf = 0,
\end{equation*} together with the fact that no current leaves the medium
\begin{equation*}
\Jbf \cdot \nubf = 0 \quad \text{ on }\dr \Omega.
\end{equation*} 
Since $v$ is a solution to (\ref{eq:V}), 
$v$ has to be constant. So, in order to reconstruct $\Jbf$ one just needs to reconstruct $w$.

\subsubsection{Recovery of $\Jbf$}

Under the assumption $|\Bbf_1| \ll |\Bbf_0|$ in $\Omega \times \mathbb{R}_+$ and 
$|\delta \rho| \ll \rho_0$ in $\Omega$,
the pressure source term $f$ defined by (\ref{deff}) can 
be approximated as follows:
$$
f(x)\approx \frac{1}{\rho_0}  
\g\cdot(  \Jbf(x) \times \Bbf_0) (u(T_{\text{pulse}}) - u(0)),
$$
where we have used that $\Jbf(x,t) = \Jbf(x) u^\prime(t)$.  

Since $\Bbf_0$ is constant  we get
\begin{equation*}
\g\cdot(  \Jbf(x) \times \Bbf_0)= \left(\g \times \Jbf\right) \cdot \Bbf_0=\vert \Bbf_0 \vert \text{curl }\Jbf.
\end{equation*}
Now, since $\Bbf_0$ is known, we can compute $w$ as the unique solution of 
\begin{equation}\label{eq:defw}
\left\{ \begin{aligned}
-\Delta w =\frac{\rho_0 f}{\vert \Bbf_0\vert (u(T_{\text{pulse}}) - u(0))}  \quad &\text{ in } \Omega,\\
w = 0 \quad &\text{ on } \dr \Omega,
\end{aligned}\right.
\end{equation} and then, by Proposition \ref{prop:Jdecomposition}, compute $\Jbf$ by $\Jbf = \textbf{curl } w$.

Note that since the problem is reduced to the two dimensional case, $\Jbf$ is then contained in the plane $\Bbf_0^\top$ with $\top$ denoting the orthogonal.


%

\subsection{Recovery of the conductivity from internal electric current density}
In this subsection we denote by $\sigma_\star$ the true conductivity of the medium, and we assume that $\sigma_\star \in L^\infty_{a,b} (\Omega)$ with $0<a<b$, i.e., it is bounded from below and above by positive known constants and is equal to some given positive constant $\sigma_0$ in a neighborhood  of $\partial \Omega$.

\subsubsection{Optimal control method}\label{sec:optimal}

Recall that $\Abf_1$ is defined by $\nabla \cdot \Abf_1 =0, B_1(x) \mathbf{e}_3 = \nabla \times \Abf_1(x)$.  Define the following operator $\mathcal{F}$:
\begin{equation*}
\begin{aligned}
L^\infty_{a,b}(\Omega) &\longrightarrow H^1(\Omega) \\
\sigma &\longmapsto  \mathcal{F}[\sigma] 
\end{aligned}
\end{equation*}
 with 
 \begin{equation} \label{defv}
 \mathcal{F}[\sigma] := U \left\{ \begin{aligned}
 \g\cdot \sigma\g U =& - \g\cdot\sigma {\Abf_1} \quad &\text{in } \Omega,\\
\sigma \frac{\dr U}{\dr \nu} = & - \sigma  \Abf_1\cdot \nubf \quad & \text{on } \dr \Omega,\\
\int_{\Omega} U=& 0.
 \end{aligned}\right.
 \end{equation}
 The following lemma holds.
 \begin{lemm}
 The operator $\mathcal{F}$ is Fr\'echet differentiable. For any $\sigma \in L^\infty_{a,b} (\Omega)$ and $h$  such that $\sigma+h \in L^\infty_{a,b} (\Omega)$, we have
 \begin{equation}\label{eq:defdF}
  d\mathcal{F}[\sigma](h) := q\left\{ \begin{aligned}
 \g\cdot \sigma\g q =& - \g\cdot h {\Abf_1}- \g \cdot h \g\mathcal{F}[\sigma] \quad &\text{in } \Omega,\\
\sigma \frac{\dr q}{\dr \nu} = &0\quad & \text{on } \dr \Omega,\\
\int_{\Omega} q=& 0.
 \end{aligned}\right.
 \end{equation}
 \end{lemm}

 \begin{proof} Denote by $r$ the function $\mathcal{F}[\sigma +h] - \mathcal{F}[\sigma] - q$. The function $r$ belongs to $H^1(\Omega)$ and satisfies the following equation in $\Omega$:
 \begin{equation*}
 \g \cdot \sigma \g r =\g \cdot h\g\left( \mathcal{F}[\sigma] - \mathcal{F}[\sigma +h]\right),
 \end{equation*}
 together with the boundary condition 
 \begin{equation*}
 \frac{\dr r}{\dr \nu} =0 \quad \text{on } \dr \Omega,
 \end{equation*}  and the zero mean condition $\int_\Omega r =0$. We have the following estimate: 
 \begin{equation*}
 \Vert \g r \Vert_{L^2(\Omega)} \leq \frac{1}{a} \Vert h \Vert_{L^\infty(\Omega)} \Vert\g\left( \mathcal{F}[\sigma] - \mathcal{F}[\sigma +h]\right)\Vert_{L^2(\Omega)}.
 \end{equation*}
 Since  $ \mathcal{F}[\sigma] - \mathcal{F}[\sigma +h]$ satisfies \begin{equation*}
 \g \cdot \left(\sigma \g\left( \mathcal{F}[\sigma] - \mathcal{F}[\sigma +h]\right)\right)= -\g \cdot \left(h \g \mathcal{F}[\sigma+h]\right) + \g \cdot\left( h { \Abf_1} \right) 
 \end{equation*} with the boundary condition
 \begin{equation*}
 \frac{\dr}{\dr \nu} \left(\mathcal{F}[\sigma+h] -\mathcal{F}[\sigma]\right)=0, 
 \end{equation*} and the zero mean condition  $\int_{\Omega} \left(\mathcal{F}[\sigma+h] -\mathcal{F}[\sigma]\right)=0$. We can also estimate the $L^2$-norm 
 of $\g \left(\mathcal{F}[\sigma+h] -\mathcal{F}[\sigma]\right)$ as follows:
 \begin{equation*}
 \Vert \g \left(\mathcal{F}[\sigma+h] -\mathcal{F}[\sigma]\right) \Vert_{L^2(\Omega)} \leq \frac{1}{a} \Vert h \Vert_{L^\infty(\Omega)} \left(\Vert \g \mathcal{F}[\sigma+h] \Vert_{L^2(\Omega)} + \Vert {\Abf}_1\Vert_{L^2(\Omega)}\right).
 \end{equation*}
 Therefore, we can bound the $H^1$-norm of $\mathcal{F}[\sigma +h]$ independently of $\sigma$ and $h$ for $||h||_{L^\infty}$ small enough. There exists a constant $C$, depending only on $\Omega$, $a$, $b$, and ${\Abf}_1$, such that
 \begin{equation*}
\Vert \g \mathcal{F}[\sigma+h] \Vert_{L^2(\Omega)} \leq C.
 \end{equation*}
 Hence, we get
 \begin{equation*}
 \Vert \g \left(\mathcal{F}[\sigma+h] -\mathcal{F}[\sigma]\right) \Vert_{L^2(\Omega)} \leq \frac{1}{a} \Vert h \Vert_{L^\infty(\Omega)}\left(C+\Vert { \Abf}_1 \Vert_{L^2(\Omega)}\right),
 \end{equation*}
 and therefore, 
 \begin{equation*}
 \Vert \g r \Vert_{L^2(\Omega)} \leq \widetilde{C} \Vert h \Vert_{L^\infty (\Omega)}^2,
 \end{equation*}
 which shows the Fr\'echet differentiability of $\mathcal{F}$.  
 \end{proof}
 
 Now, we introduce the misfit functional:
 
 \begin{equation}\label{eq:defj}
 \begin{aligned}
L^\infty_{a,b} &\longrightarrow \mathbb{R} \\
\sigma &\longmapsto \mathcal{J}[\sigma]= \frac{1}{2}\int_\Omega \vert \sigma \left(\g \mathcal{F}[\sigma] + { \Abf}_1\right) - \Jbf \vert^2,
\end{aligned}
 \end{equation}
 \begin{lemm}\label{lem:difJ}
 The misfit functional $\mathcal{J}$ is Fr\'echet-differentiable. For any $\sigma\in L^\infty_{a,b}(\Omega)$, we have
 \begin{equation*}
 d\mathcal{J}[\sigma]=\left(\sigma \g \mathcal{F}[\sigma] + \sigma {\Abf}_1- \Jbf  \right)\cdot \left( \g \mathcal{F}[\sigma]  + {\Abf}_1 \right) +   \g s \cdot \left({\Abf}_1 + \g \mathcal{F}[\sigma] \right),
 \end{equation*}
 where $s$ is defined as the solution to the adjoint problem:
  \begin{equation}\label{eq:defs}
\left\{ \begin{aligned}
 \g\cdot \sigma\g s=& \g\cdot \left(\sigma^2 \g \mathcal{F}[\sigma] + \sigma^2 {\Abf}_1- \sigma\Jbf\right) \quad &\text{in } \Omega,\\
\sigma \frac{\dr s}{\dr \nu} = &0\quad & \text{on } \dr \Omega,\\
\int_{\Omega} s=& 0.
 \end{aligned}\right.
 \end{equation}
 
 \end{lemm}
 
 \begin{proof} Since $\mathcal{F}$ is Fr\'echet-differentiable, so is $\mathcal{J}$. For any $\sigma \in L^\infty_{a,b}(\Omega)$ and $h$ such that $\sigma +h \in L^\infty_{a,b}(\Omega)$, we have
 \begin{equation*}
 d\mathcal{J}[\sigma](h)=\int_\Omega\left(\sigma \g \mathcal{F}[\sigma] + \sigma {\Abf}_1 - \Jbf\right)\cdot \left(h\g( \mathcal{F}[\sigma] + {\Abf}_1 )+\sigma \g (d\mathcal{F}[\sigma](h)) \right).
 \end{equation*}
 Multiplying (\ref{eq:defs}) by $d\mathcal{F}[\sigma](h)$ we get
 \begin{equation*}
 \int_\Omega  \sigma \left(\sigma \g \mathcal{F}[\sigma] + \sigma {\Abf}_1 - \Jbf\right) \cdot \g d\mathcal{F}[\sigma](h)= \int_\Omega \sigma \g s \cdot \g d\mathcal{F}[\sigma](h).
 \end{equation*}
 On the other hand, multiplying (\ref{eq:defdF}) by $s$ we obtain
  \begin{equation*}
 \int_\Omega \sigma \g s \cdot \g d\mathcal{F}[\sigma](h)= \int_\Omega  h \g s \cdot \left({\Abf}_1 + \g \mathcal{F}[\sigma] \right) .
 \end{equation*}
 So we have
 \begin{equation*}
  d\mathcal{J}[\sigma](h)=\int_\Omega h\bigg[\left(\sigma \g \mathcal{F}[\sigma] + \sigma {\Abf}_1 - \Jbf  \right)\cdot \left( \g \mathcal{F}[\sigma]  + {\Abf}_1 \right) +  \g s \cdot \left({\Abf}_1 + \g \mathcal{F}[\sigma] \right) \bigg],
 \end{equation*}
and the proof is complete. 
 \end{proof}
 
Lemma  \ref{lem:difJ} allows us to apply the gradient descent method in order to minimize the discrepancy functional
$\mathcal{J}$. Let $\sigma_{(0)}$ be an initial guess. We compute the iterates
\begin{equation} \label{defeta}
\sigma_{(n+1)} = T[\sigma_{(n)}] - \mu d\mathcal{J}[ T[\sigma_{(n)}]], 
\quad \forall n\in \mathbb{N},
\end{equation}
where $\mu >0$ is the step size and
$T[f]= \min \{ \max \{ f, a \}, b\}$.

In the sequel, we prove the convergence of (\ref{defeta}) with two excitations. 
Let $\Jbf^{(1)}$ and $\Jbf^{(2)}$ correspond to two different excitations $\Abf_1^{(1)}$ and 
$\Abf_1^{(2)}$. Assume that $\Jbf^{(1)} \times \Jbf^{(2)} \neq 0$ in $\Omega$. Let $\mathcal{G}^{(i)}: \sigma \mapsto \sigma \g\left( \mathcal{F}^{(i)}[\sigma] + { \Abf}^{(i)}_1\right) - \Jbf_i$, where $\mathcal{F}^{(i)}$ is defined by (\ref{eq:defdF}) with 
$\Abf_1= \Abf_1^{(i)}$ for $i=1,2$. The optimal control algorithm (\ref{defeta}) with two excitations is equivalent to the following Landweber scheme given by
\begin{equation} \label{defeta2}
\sigma_{(n+1)} = T[\sigma_{(n)}] - \mu d\mathcal{G}^\star[ \mathcal{G}[T[\sigma_{(n)}]]], 
\quad \forall n\in \mathbb{N},
\end{equation}
where $\mathcal{G}[\sigma] = (\mathcal{G}^{(1)}[\sigma], \mathcal{G}^{(2)}[\sigma])^T$.

Following \cite{laure2}, we prove the convergence and stability of (\ref{defeta2}) provided that two magnetic excitations leading to nonparallel current densities are employed. 

\begin{prop} Let $\Jbf^{(1)}$ and $\Jbf^{(2)}$ correspond to two different excitations. Assume that $\Jbf^{(1)} \times \Jbf^{(2)} \neq 0$ in $\Omega$. Then there exists $\eta >0$ such that if $|| \sigma_{(0)} - \sigma_\star||_{H^1_0(\Omega)} \leq \eta$, then $|| \sigma_{(n)} - \sigma_\star||_{H^1_0(\Omega)} \rightarrow 0$ as $n\rightarrow + \infty$. 

\end{prop}
\begin{proof} According to  \cite{laure2}, it suffices to prove that there exists a positive constant $C$ such that
\begin{equation}
\label{J1J2}
|| d\mathcal{G}[\sigma](h)||_{H^1(\Omega)} \geq C ||h||_{H^1_0(\Omega)}
\end{equation}
for all $h \in H^1_0(\Omega)$ such that $\sigma + h \in L^\infty_{a,b}(\Omega)$. 
We have
$$
d\mathcal{G}^{(i)}[\sigma](h)= h \Jbf^{(i)} + \sigma \nabla d\mathcal{F}^{(i)}[\sigma](h).
$$
Therefore,
$$
\nabla\cdot d\mathcal{G}^{(i)}[\sigma](h) =0, \quad d\mathcal{G}^{(i)}[\sigma](h)\cdot \nubf =0,
$$
and 
$$
\nabla \times (\frac{1}{\sigma} d\mathcal{G}^{(i)}[\sigma](h)) = h \nabla \times (\frac{1}{\sigma} \Jbf^{(i)})  +\sigma \nabla h \times \Jbf^{(i)}.$$
Since $\nabla \times (\frac{1}{\sigma} \Jbf^{(i)}) \times \mathbf{e}_3=0$ and  $\Jbf^{(1)} \times \Jbf^{(2)} \neq 0$, it follows that 
$$
|| h||_{H^1_0(\Omega)} \leq C \sum_{i=1}^2 ||d\mathcal{G}^{(i)}[\sigma](h)||_{H^1(\Omega)},
$$
which completes the proof. 
 \end{proof} 
Let $\mathcal{F}[\sigma] = (\mathcal{F}^{(1)}[\sigma], \mathcal{F}^{(2)}[\sigma])^T$. Note that analogously to (\ref{J1J2}) there exists a positive constant $C$ such that
$$
|| d\mathcal{F}[\sigma](h)||_{H^1(\Omega)} \geq C ||h||_{H^1_0(\Omega)}
$$
for all $h \in H^1_0(\Omega)$ such that $\sigma + h \in L^\infty_{a,b}(\Omega)$, provided that $\Jbf^{(1)} \times \Jbf^{(2)} \neq 0$ in $\Omega$.

\subsubsection{Fixed point method}\label{sec:fixed}
 
In this subsection, we denote by $\sigma_\star$ the true conductivity inside the domain $\Omega$. We also make the following assumptions:
\begin{itemize}
\item $\ds \exists c>0 , \quad \text{such that }\vert \Bbf_1 \vert >c \quad \text{ in } \Omega$;
\item $\ds\sigma \in \mathcal{C}^{0,\alpha}(\overline{\Omega}), \quad \alpha\in ]0,1[;$
\item $\sigma_\star = \sigma_0$ in an open neighborhood of $\dr \Omega$.
\end{itemize}

 From the unique continuation principle, the following lemma holds.
 \begin{lemm} The set $\{x\in \Omega, \Jbf(x) =0\}$ is nowhere dense.
 \end{lemm}

 The interior data is $\Jbf = \sigma_\star \left[\g \mathcal{F}[\sigma_\star]  + {\Abf}_1 \right]$. One can only hope to recover $\sigma_\star$ at the points where $\Jbf \neq 0$. Even then, we can expect any type of reconstruction to be numerically unstable in sets where $\Jbf$ is very small. Assume that  $\Jbf $ is continuous and let $\varepsilon>0$ and $x_0$ be such that $\vert \Jbf (x_0) \vert > 2 \varepsilon$. 
 We define $\Omega_\varepsilon$ to be a neighborhood of $x_0$ such that for any $x\in \Omega_\varepsilon$, $
 \vert \Jbf (x) \vert > \varepsilon.$
One can assume that $\Omega_\varepsilon$ is a $\C^1$ domain without loosing generality. Now, introduce the operator $\mathcal{F}_\varepsilon$ as follows:
\begin{equation*}
\begin{aligned}
L^\infty_{a,b}(\Omega_\varepsilon) \longrightarrow & H^1(\Omega_\varepsilon) \\
\sigma \longmapsto & \mathcal{F}_\varepsilon[\sigma] := V_\varepsilon,
\end{aligned}
\end{equation*}
where $V_\varepsilon$ satisfies the following equation:
  \begin{equation}\label{eq:defVeps}
\left\{ \begin{aligned}
 \g\cdot \sigma\g V_\varepsilon=&- \g\cdot \left(\sigma {\Abf}_1 \right) \quad &\text{in } \Omega_\varepsilon,\\
 \sigma \frac{\dr V_\varepsilon}{\dr \nu} = &-\sigma {\Abf}_1 \cdot \nubf + \Jbf\cdot \nubf \quad & \text{on } \dr \Omega_\varepsilon,\\
\int_{\Omega_\varepsilon} V_\varepsilon=& 0,
 \end{aligned}\right.
 \end{equation}
where $\nubf$ denotes the outward normal to $\partial \Omega_\varepsilon$. Note that 
$\int_{\partial \Omega_\varepsilon} \Jbf \cdot \nubf = 0$ since $\nabla \cdot \Jbf=0 $ in $\Omega_\varepsilon$. 

We also define the nonlinear operator $\mathcal{G}_\varepsilon$ by
\begin{equation}\label{eq:defG}
\begin{aligned}
L^\infty_{a,b}(\Omega_\varepsilon) \longrightarrow & L^\infty(\Omega_\varepsilon) \\
\sigma \longmapsto & \mathcal{G}_\varepsilon[\sigma] :=\sigma \frac{\left(\sigma \g V_\varepsilon[\sigma] + \sigma {\Abf}_1 \right) \cdot \Jbf}{\left\vert \Jbf \right\vert^2}.
\end{aligned}
\end{equation}

 \begin{lemm}\label{lem:fixedpoint}
 The restriction of $\sigma_\star$ on ${\Omega_\varepsilon}$ is a fixed point for the operator $\mathcal{G}_\varepsilon$.

 \end{lemm}
\begin{proof}  For the existence it suffices to prove that \mbox{$\mathcal{F}_\varepsilon\left[\sigma_\star \vert_{\Omega_\varepsilon}\right] = \mathcal{F}[\sigma_\star]\big\vert_{\Omega_\varepsilon}$.} Denote by $V_\star = \mathcal{F}[\sigma_\star]$. We can see that $V_\star$ satisfies
\begin{equation*}
\g\cdot \sigma_\star \g V_\star = - \g\cdot \left(\sigma {\Abf}_1 \right) \quad \text{in } \Omega_\varepsilon.
\end{equation*}
Taking the normal derivative along the boundary of $\Omega_\varepsilon$, we get
\begin{equation*}
\sigma \frac{\dr V_\star}{\dr \nu} = -\sigma {\Abf}_1 \cdot \nubf + \Jbf\cdot \nubf \quad \text{on } \dr \Omega_\varepsilon.
\end{equation*} 
From the well posedness of (\ref{eq:defVeps}), it follows  that
 \begin{equation*}
V_\star \big\vert_{\Omega_\varepsilon} = \mathcal{F}_\varepsilon[\sigma_\star\big\vert_{\Omega_\varepsilon}  ] + c, \quad c\in \R.
\end{equation*} 
So, we arrive at \begin{equation*}
\mathcal{G}_\varepsilon\left[\sigma_\star \big\vert_{\Omega_\varepsilon}\right]= \sigma_\star \big\vert_{\Omega_\varepsilon}.
\end{equation*}
%
%
\end{proof}

We  need the following lemma. We refer to \cite{sohr2012navier} for its proof.
\begin{lemm}\label{lem:divv}
Let $\Omega\subset \R^2$ be a bounded domain with Lipschitz boundary. For each $g\in H^{-1}(\Omega)$ there exists at least one $\mathbf{v}\in L^2(\Omega)$ with $\g\cdot \mathbf{v} =g$ in the sense of the distributions and
\begin{equation*}
\Vert \mathbf{v} \Vert_{L^2(\Omega)} \leq C \Vert g\Vert_{H^{-1}(\Omega)}
\end{equation*} with the constant $C$ depending only on $\Omega$.
\end{lemm}

The following result holds.
\begin{lemm}\label{lem:contractant} If $\Vert {\Abf}_1 \Vert_{L^2(\Omega_\varepsilon)}$ is small enough, then the operator $\mathcal{G}_\varepsilon$ is a contraction.
\end{lemm}

\begin{proof} Take $\sigma_1$ and $\sigma_2$ in $L^\infty_{a,b} (\Omega)$. We have
\begin{multline*}
\left\vert \mathcal{G}_\varepsilon[\sigma_1](x) -\mathcal{G}_\varepsilon[\sigma_2](x) \right\vert = \frac{1}{\vert\Jbf(x)\vert^2}  \\ \times \left\vert \left(\sigma_1^2(x)\g V_\varepsilon[\sigma_1](x) -  \sigma_2^2(x)\g V_\varepsilon[\sigma_2](x) + \left(\sigma_1^2(x)- \sigma_2^2(x)\right) {\Abf}_1(x)\right)\cdot \Jbf(x) \right\vert,
\end{multline*}
which gives, using the Cauchy-Schwartz inequality:
\begin{multline*}
\left\vert \mathcal{G}_\varepsilon[\sigma_1](x) -\mathcal{G}_\varepsilon[\sigma_2](x) \right\vert \leq \frac{1}{\varepsilon}  \\  \times\left\vert \left(\sigma_1^2(x)\g V_\varepsilon[\sigma_1](x) -  \sigma_2^2(x)\g V_\varepsilon[\sigma_2](x) + \left(\sigma_1^2(x)- \sigma_2^2(x)\right) {\Abf}_1(x)\right) \right\vert .
\end{multline*}
The right-hand side can be rewritten using the fact that $\vert\sigma_i(x)\vert \leq b$ for $i=1,2$, and hence, 
\begin{multline}\label{eq:gx1-gx2}
\left\vert \mathcal{G}_\varepsilon[\sigma_1](x) -\mathcal{G}_\varepsilon[\sigma_2](x) \right\vert \leq \frac{b}{\varepsilon}  \\ \times \left[\left\vert \sigma_1(x)\g V_\varepsilon[\sigma_1](x) -  \sigma_2(x)\g V_\varepsilon[\sigma_2](x) \right\vert + \left\vert \left( \sigma_1(x)- \sigma_2(x) \right) {\Abf}_1(x)\right\vert \right] .
\end{multline}
Now, consider the function $\mathbf{v} = \sigma_1\g V_\varepsilon[\sigma_1] -  \sigma_2\g V_\varepsilon[\sigma_2]$. We get
\begin{equation*}
\g \cdot \mathbf{v} = - \g\cdot \left[\left(\sigma_1- \sigma_2\right) {\Abf}_1\right] \quad \text{in } \dr \Omega_\varepsilon,
\end{equation*}
along with the boundary condition $\mathbf{v} \cdot \nubf = 0$ on $\dr \Omega_\varepsilon$.
Using Lemma \ref{lem:divv}, there exists a constant $C$ depending only on $\Omega_\varepsilon$ such that
\begin{equation*}
\Vert \mathbf{v}\Vert_{L^2(\Omega_\varepsilon)} \leq C \Vert \g\cdot \left[\left(\sigma_1- \sigma_2\right) {\Abf}_1\right]  \Vert_{H^{-1}(\Omega_\varepsilon)},
\end{equation*}
which shows that
\begin{equation*}
\Vert \mathbf{v}\Vert_{L^2(\Omega_\varepsilon)} \leq C \Vert \left(\sigma_1- \sigma_2\right) {\Abf}_1\  \Vert_{L^2(\Omega_\varepsilon)}.
\end{equation*}
Using Cauchy-Schwartz inequality:
\begin{equation}\label{eq:normv}
\Vert \mathbf{v}\Vert_{L^2(\Omega_\varepsilon)} \leq C \Vert \sigma_1- \sigma_2\Vert_{L^2(\Omega_\varepsilon)} \Vert {\Abf}_1\  \Vert_{L^2(\Omega_\varepsilon)}.
\end{equation}
Putting together (\ref{eq:gx1-gx2}) with (\ref{eq:normv}), we arrive at
\begin{equation*}
\left\Vert \mathcal{G}_\varepsilon[\sigma_1] -\mathcal{G}_\varepsilon[\sigma_2]\right\Vert_{L^2(\Omega_\varepsilon)} \leq (C+1)\frac{b}{\varepsilon}\Vert {\Abf}_1 \Vert_{L^2(\Omega_\varepsilon)}  \Vert \sigma_1- \sigma_2\Vert_{L^2(\Omega_\varepsilon)}.
\end{equation*}
The proof is then complete.

\end{proof}

The following proposition shows the convergence of the fixed point reconstruction algorithm.
\begin{prop}
Let $\sigma_{(n)}\in \left(L^2(\Omega_\varepsilon)\right)^\mathbb{N}$ be the sequence defined by 
\begin{equation}
\begin{aligned}
&\sigma_{(0)}=1,\\
&\sigma_{(n+1)}=\max\left( \min\left( \mathcal{G}_\varepsilon[\sigma_{(n)}],b\right), a\right), \quad \forall n\in \mathbb{N}.
\end{aligned}
\end{equation}
If $\Vert {\Abf }_1 \Vert_{L^2(\Omega_\varepsilon)}$ is small enough, then the sequence is well defined and $\sigma_{(n)}$ converges to $\sigma_\star\big\vert_{\Omega_\varepsilon}$ in $L^2(\Omega_\varepsilon)$.
\end{prop}

\begin{proof}  Let $(X,d) =\left( L^\infty_{a,b}(\Omega_\varepsilon), \Vert \cdot \Vert_{L^2(\Omega_\varepsilon)}\right)$. Then, $(X,d)$ is a complete, non empty metric space. Let ${T}_\varepsilon$ be the map defined by
\begin{equation*}
\begin{aligned}
L^\infty_{a,b}(\Omega_\varepsilon) \longrightarrow & L^\infty_{a,b}(\Omega_\varepsilon) \\
\sigma \longmapsto & {T}_\varepsilon[\sigma] := \max\left( \min\left( \mathcal{G}_\varepsilon[\sigma],b\right), a\right).
\end{aligned}
\end{equation*}
Using Lemma \ref{lem:contractant}, we get that ${T}_\varepsilon$ is a contraction, provided that $\Vert {\Abf}_1 \Vert_{L^2(\Omega_\varepsilon)}$ is small enough. We already have the existence of a fixed point given by Lemma \ref{lem:fixedpoint}, and therefore, Banach's fixed point theorem gives the convergence of the sequence for the $L^2$ norm over $\Omega_\varepsilon$, and the uniqueness of the fixed point.
\end{proof}

\subsubsection{Orthogonal field method}\label{sec:ortho}

In this section we present a non-iterative method to reconstruct the electrical conductivity from the electric current density.  This direct method was first introduced in \cite{ammari2014mathematical} and works with piecewise regularity for the true conductivity $\sigma_\star$ in the case of a Lorentz force electrical impedance tomography experiment. However, the practical conditions are a bit different here and we have to modify the method to make it work in the present case. 

We assume in this section that $\ds\sigma_\star \in \mathcal{C}^{0,\alpha}(\overline{\Omega}), \quad \alpha\in ]0,1].$ 
The fields $\Jbf=(J_1,J_2)$ and ${\Abf}_1$ are assumed to be known in $\Omega$.  Our goal is to reconstruct $V_\star$ the solution of  (\ref{eq:V}) in $H^1(\Omega)$. Then, computing $\frac{\vert \g V_\star + \Abf_1 \vert}{\vert \Jbf \vert}$ for  $\vert \Jbf\vert $ nonzero will give us $\frac{1}{\sigma_\star}$.
Recall that $\Jbf=\textbf{curl } w$ where $w$ is defined by equation (\ref{eq:defw}).
\begin{de}
We say that the data $f$  on the right hand side of (\ref{eq:defw}) is admissible if $f>0$ or $f<0$ in  $\Omega$  and if the critical points of $w$ are isolated.
\end{de}

Introduce $\Fbf=(-J_2,J_1)^T$ the rotation of $\Jbf$ by $\frac{\pi }{2}$.
It is worth noticing that the true electrical potential $V_\star$ is a solution of 
\begin{equation}\label{eq:Uorthogonal}
\left\{ \begin{aligned}
\Fbf \cdot \g V_\star &= - \Fbf \cdot {\Abf}_1 \quad &\text{in } \Omega&,\\
\frac{\dr V_\star}{\dr \nu} &=0\quad  &\text{on } \dr \Omega&,\\
\int_\Omega V_\star &= 0.
\end{aligned} \right.
\end{equation}Equation (\ref{eq:Uorthogonal}) has a unique solution in $H^1(\Omega)$, and this solution is  the true potential $V_\star$.


The following uniqueness result holds.
\begin{prop}\label{prop:unicity} If $U \in H^1(\Omega)$ is a solution of
\begin{equation}\label{eq:Uorthogonalhomogeneous}
\left\{ \begin{aligned}
\Fbf \cdot \g U&= 0 \quad &\text{in } \Omega,\\
\frac{\dr U}{\dr \nu} &=0\quad  &\text{on } \dr \Omega &,\\
\int_\Omega U &= 0,
\end{aligned} \right.
\end{equation}
then $U=0$ in $\Omega$.
\end{prop}
\begin{proof} We use the characteristic method (see, for instance, \cite{evans2010partial}) for solving (\ref{eq:Uorthogonalhomogeneous}). 
For any $x_0\in \Omega$, consider the Cauchy problem:

 \begin{equation}\label{eq:cauchyproblem}
 \left\{ \begin{aligned}
&\frac{d X}{dt} = \Fbf\left(X(t)\right),  \quad t\in \R,\\
& X(0)=x_0 \in \Omega.
 \end{aligned}\right.
 \end{equation} 
We call the set $\{ x(t),t\in \R\}$ the integral curve at $x_0$.
 Since $\sigma \in \mathcal{C}^{0,\alpha}(\Omega)$, $\Fbf \in \mathcal{C}^{1,\alpha}(\Omega)$. Then, we can apply the Cauchy-Lipschitz theorem and get global existence and uniqueness of a solution to (\ref{eq:cauchyproblem}).
Denote by $T$ the upper bound on the domain size of integral curves.
Now, assume that $U \in H^1(\Omega)$ is a solution of (\ref{eq:Uorthogonalhomogeneous}).
Since $\Jbf= \textbf{curl } w$, $\Fbf$ can be written as
\begin{align*}
\Fbf= -\nabla w \quad \text{ in \  }\Omega.
\end{align*}
Equation (\ref{eq:cauchyproblem}) reduces to the following gradient flow problem:
 \begin{equation}\label{eq:gradientflow}
 \left\{ \begin{aligned}
&\frac{d X}{dt} = -\nabla w \left(X(t)\right),  \quad t\in \R,\\
& X(0)=x_0 \in \Omega.
 \end{aligned}\right.
 \end{equation} 
Using \cite{alessandrini1992index}, we know that there are finitely many isolated critical points $p_1,\ldots, p_n$, for $w$ on $\Omega$. It is also known (see \cite[p. 204]{hirsch73differential}) that since the sets $w^{-1}\left(]-\infty,c]\right) $ are compact for every $c\in \mathbb{R}$,  $\lim_{t\rightarrow \infty} X(t)$ exists and is equal to one of the equilibrium points $p_1,\ldots, p_n$.
Now, for every $i$, we define $\Omega_i$ the set of points $x_0 \in \Omega$ such that the solution of (\ref{eq:gradientflow})  converges to $p_i$. Therefore, we have $\Omega=\cup_{i=1}^n \Omega_i$. 

Now, for any $i$ consider $x_0\in \Omega_i$, and $X\in \mathcal{C}^1\left([0,T[,\Omega\right)$ the solution of (\ref{eq:gradientflow}). 
We define $f\in \mathcal{C}^0\left(\R^+,\R\right)$ by $f(t)=U(X(t))$. The function $f$ is differentiable on $\R^+$ and $f'(t)=\g U(X(t))\cdot \Fbf(X(t))=0$. Hence, $f$ is  constant. We have
\begin{equation*}
U(x_0)=f(0)=\lim_{t\rightarrow \infty} f(t)=U(p_i)=c_i \in \mathbb{R}.
\end{equation*}
So, $U$ is constant equal to $c_i$ in $\Omega_i$. The regularity of $U$ implies that $\forall i,j \in \llbracket 1,n\rrbracket,$ \mbox{$ c_i=c_j.$ } Therefore $U$ is constant on $\Omega$ and the zero integral condition yields
\begin{equation*}
U=0 \text{ in } \Omega.
\end{equation*}
This shows the uniqueness of a solution  to ($\ref{eq:Uorthogonalhomogeneous}$) and thus, concludes the proof. \end{proof}

In order to solve numerically  (\ref{eq:Uorthogonal}), we use a method of vanishing viscosity \cite{ammari2014mathematical}.
The field $\Abf_1$ is known and we can solve uniquely the following problem:
\begin{equation}\label{eq:Neumannregularized}
\left\{ \begin{aligned}
\g \cdot \left[\left(\eta  I+ \Fbf\Fbf^T\right) \g U^{(\eta)}\right] &= - \g \cdot 
\Fbf \Fbf^T {\Abf}_1 \quad &\text{in } \Omega&,\\
\frac{\dr U^{(\eta)} }{\dr \nu}&= - {\Abf}_1 \cdot \nubf \quad  &\text{on } \dr \Omega&,\\
\int_\Omega U^{(\eta)} &=0,
\end{aligned} \right.
\end{equation} for some small $\eta >0$. Here, $I$ denotes the $2\times 2$ identity matrix. 

\begin{prop} \label{propvis} Let $\sigma_\star$ be the true conductivity. Let $V_\star$ be the solution to (\ref{defv}) with $\sigma =\sigma_\star$. 
The solution $U^{(\eta)} $ of (\ref{eq:Neumannregularized}) converges strongly to $V_\star$ in $H^1(\Omega)$ when $\eta$ goes to zero.
\end{prop}
\begin{proof}
We can easily see that $\widetilde{U}^{(\eta)}= U^{(\eta)}-V_\star$ is the solution to
\begin{equation}\label{eq:Utilde_eta}
\left\{ \begin{aligned}
\g \cdot \left[\left(\eta  I+ \Fbf\Fbf^T\right) \g\widetilde{U}^{(\eta)}\right] &=- \eta \Delta V_\star \quad &\text{in } \Omega&,\\
\frac{\dr \widetilde{U}^{(\eta)} }{\dr \nu}&= 0 \quad  &\text{on } \dr \Omega&,\\
\int_\Omega \widetilde{U}^{(\eta)} &=0.
\end{aligned} \right.
\end{equation}
Multiplying (\ref{eq:Utilde_eta}) by $\widetilde{U}^{(\eta)}$ and integrating by parts over $\Omega$, we find that
\begin{equation}\label{eq:IPPUeta}
\eta \int_\Omega \vert \g \widetilde{U}^{(\eta)}\vert^2 + \int_\Omega \vert \Fbf \cdot \g \widetilde{U}^{(\eta)} \vert^2 = \eta \int_\Omega \g \widetilde{U}^{(\eta)} \cdot \g V_\star  +\eta \int_{\dr \Omega}  \widetilde{U}^{(\eta)} {\Abf}_1 \cdot \nubf,
\end{equation}
since $\ds \frac{\dr \widetilde{U}^{(\eta)} }{\dr \nu}=0$ and $\ds \frac{\dr V_\star}{\dr \nu} = - {\Abf}_1 \cdot \nubf$. 
Therefore, we have
\begin{equation*}
\Vert \widetilde{U}^{(\eta)} \Vert^2_{H^1(\Omega)} \leq  \Vert \widetilde{U}^{(\eta)} \Vert_{H^1(\Omega} \Vert V_\star \Vert_{H^1(\Omega)} + C  \Vert \widetilde{U}^{(\eta)} \Vert_{H^1(\Omega)},
\end{equation*} where $C$ depends only on $\Omega$ and  $\Abf_1$.
This shows that the sequence $( \widetilde{U}^{(\eta)} )_{\eta>0}$ is bounded in $H^1(\Omega)$. Using Banach-Alaoglu's theorem we can extract a subsequence which converges weakly to some $u^*$ in $H^1(\Omega)$. We multiply (\ref{eq:Utilde_eta}) by $u^*$ and integrate by parts over $\Omega$ to obtain
\begin{equation*}
\int_\Omega \big( \Fbf \cdot\g \widetilde{U}^{(\eta)} \big) \big( \Fbf \cdot \g u^*\big) = \eta \left[ \int_\Omega \g V_\star \cdot \g u^* -  \int_\Omega \g \widetilde{U}^{(\eta)}  \cdot  \g u^* + \int_{\dr \Omega} u^*  {\Abf}_1 \cdot \nubf \right].
\end{equation*}
Taking the limit when $\eta $ goes to zero yields
\begin{equation*}
\Vert \Fbf \cdot \g u^*\Vert_{L^2(\Omega)} = 0.
\end{equation*}
Using Proposition \ref{prop:unicity}, we have
\begin{equation*}
u^*=0 \text{ in }  \Omega,
\end{equation*}
since $u^*$ is a solution to (\ref{eq:Uorthogonalhomogeneous}).

Actually, we can see that there is no need for an extraction, since $0$ is the only accumulation point  for $\widetilde{U}^{(\eta)}$ with respect to the weak topology. If we consider a subsequence $\widetilde{U}^{(\phi(\eta))}$, it is still bounded in $H^1(\Omega)$ and therefore, using the same argument as above, zero is an accumulation point of this subsequence.
For the strong convergence, we use (\ref{eq:IPPUeta}) to get
\begin{equation} \label{aboveeq}
\int_\Omega  \vert \g \widetilde{U}^{(\eta)}\vert^2 \leq \int_\Omega \g \widetilde{U}^{(\eta)} \cdot \g V_\star  + \int_{\dr \Omega}  \widetilde{U}^{(\eta)} {\Abf}_1 \cdot \nubf.
\end{equation}
Since $\widetilde{U}^{(\eta)} \rightharpoonup 0$, the right-hand side of (\ref{aboveeq}) goes to zero when $\eta$ goes to zero. Hence,
\begin{equation*}
\Vert \widetilde{U}^{(\eta)}\Vert_{H^1(\Omega)} \longrightarrow 0 \quad \mbox{as } \eta \rightarrow 0.
\end{equation*}
\end{proof}
Now, we take $U^{(\eta)}$ to be the solution of (\ref{eq:Neumannregularized}) and define the approximated resistivity (inverse of the conductivity) by
\begin{equation} \label{defsigeta}
\frac{1}{\sigma_\eta} = \frac{\vert \g U^{(\eta)} + \Abf_1 \vert}{\vert \Jbf \vert}.
\end{equation}
Since 
$$
\frac{1}{\sigma_\star} = \frac{\vert \g V_\star + \Abf_1 \vert}{\vert \Jbf \vert},
$$
Proposition \ref{propvis} shows that $\ds \frac{1}{\sigma_\eta}$ is a good approximation for $\ds \frac{1}{\sigma_\star}$ in the $L^2$-sense.
\begin{prop} Let $\sigma_\star$ be the true conductivity and let $\sigma_\eta$ be defined by (\ref{defsigeta}). We have
\begin{equation*}
\left\Vert\frac{1}{\sigma_\eta} -\frac{1}{\sigma_\star}\right\Vert_{L^2(\Omega)} \longrightarrow 0\quad \text{as } \eta\rightarrow 0.
\end{equation*}
\end{prop}

\section{Numerical illustrations}
We set $\ds \Omega=\left\{ (x,y)\in \R^2, \ \left(\frac{x}{2}\right)^2 + y^2 <1 \right\}$. We take a conductivity $\sigma \in \mathcal{C}^{0,\alpha} (\Omega)$ as represented on Figure \ref{fig:sigma_real}.  The potential $\Abf_1$ is chosen as $$\ds \Abf_1(x)=10^{-2}\left( \frac{y}{2}+1 ; -\frac{x}{2}+1\right),$$ so that $\Bbf_1$ is constant in space.
\subsection{Optimal control}
We use the algorithm presented in section \ref{sec:optimal}. We set a step size  equal to $8\cdot 10^{-7}$ and $\sigma_{(0)} =3$ as an initial guess. After $50$ iterations, we get the reconstruction shown in Figure \ref{fig:sigma_min}. The general shape of the conductivity is recovered but the conductivity contrast is not recovered. Moreover, the convergence is quite slow. It is worth mentioning that using two nonparallel electric current densities does not improve significantly the quality of the reconstruction.

\begin{figure}
\begin{center}
\includegraphics[scale=0.3]{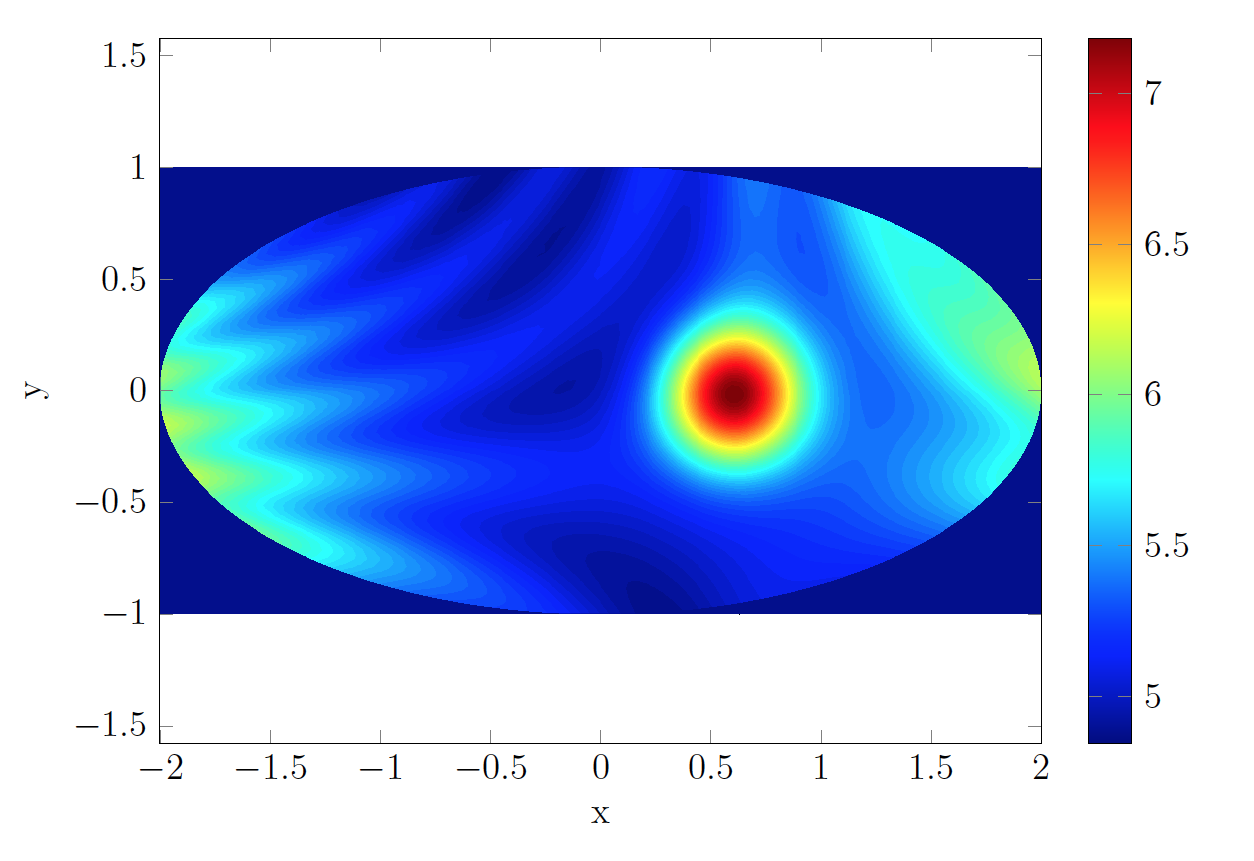}
\caption{\label{fig:sigma_real} Conductivity to be reconstructed.}
\end{center}
\end{figure}

\begin{figure}
\begin{center}
\includegraphics[scale=0.3]{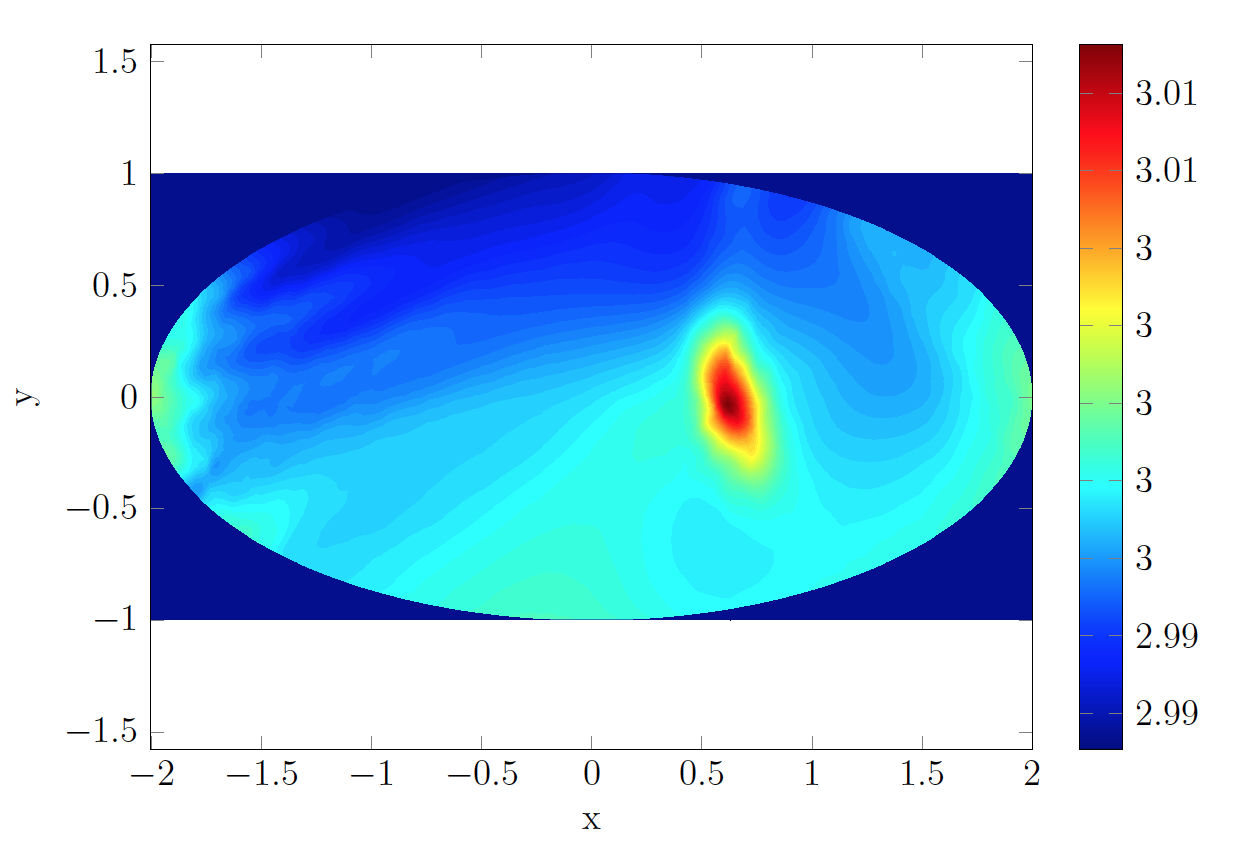}
\caption{\label{fig:sigma_min} Conductivity reconstructed by the optimal control method.}
\end{center}
\end{figure}

\subsection{Fixed-point method}
We use the algorithm described in section \ref{sec:fixed}, but slightly modified.
The operator $\mathcal{G}$ defined by \begin{equation*}
\mathcal{G}[\sigma] :=\sigma \frac{\left(\sigma \g V[\sigma] + \sigma {\Abf}_1 \right) \cdot \Jbf}{\left\vert \Jbf \right\vert^2}
\end{equation*} is replaced by 
\begin{equation*}
\widetilde{\mathcal{G}}[\sigma] :=\frac{\left(\g V[\sigma] +  {\Abf}_1 \right) \cdot \Jbf}{\left\vert \nabla V[\sigma] + {\Abf}_1\right\vert^2},
\end{equation*}
which is analytically the same but numerically is more stable.
Since the term $\left\vert \nabla V[\sigma] + {\Abf}_1\right\vert^2$ can be small,   we smooth out the reconstructed conductivity  $\sigma_{(n)}$ at each step by convolving it with a Gaussian kernel. This makes the algorithm less unstable. The result after $9$ iterations is shown in Figure \ref{fig:sigma_fixed}. The convergence is faster than the gradient descent, but the algorithm still fails at recovering the exact values of the true conductivity.

\begin{figure}
\begin{center}
\includegraphics[scale=0.3]{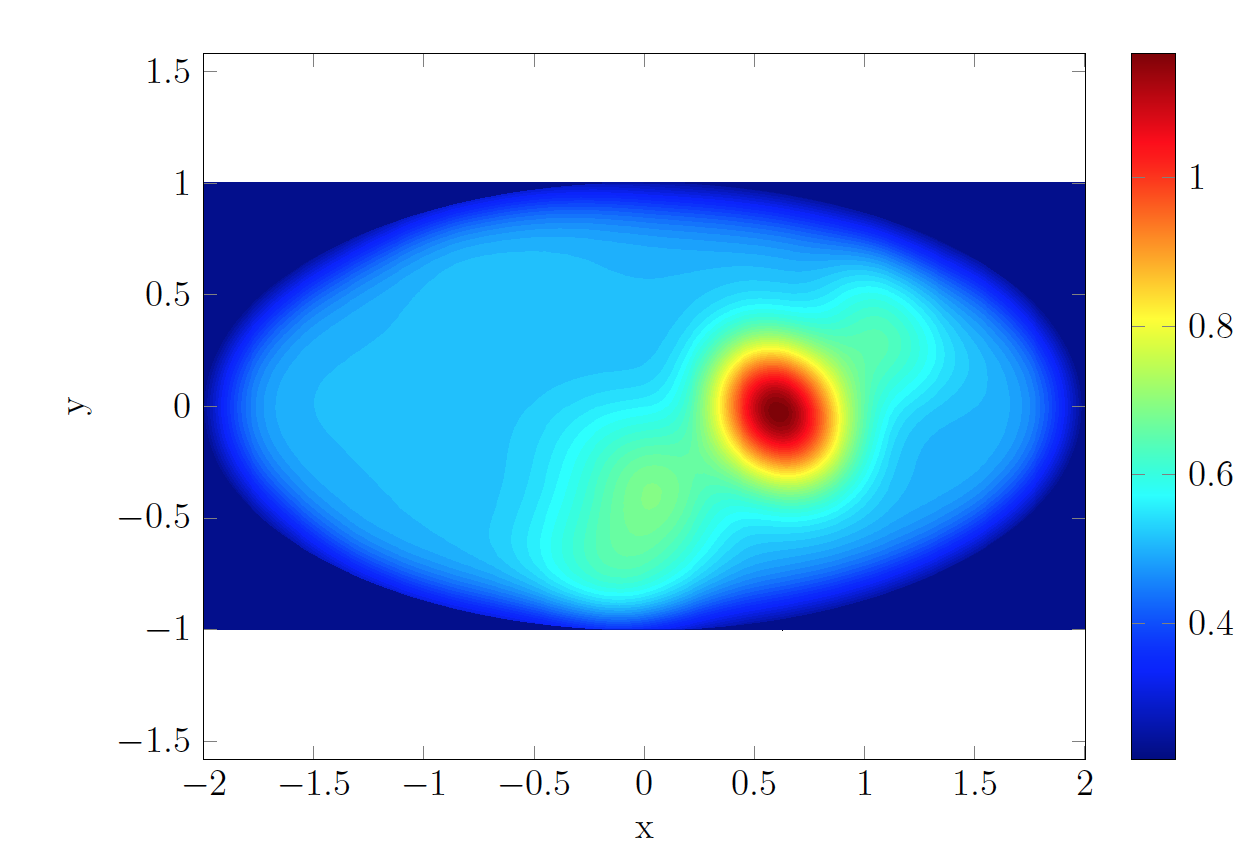}
\caption{\label{fig:sigma_fixed} Conductivity reconstructed by the fixed point method.}
\end{center}
\end{figure}

\subsection{Orthogonal field method}
We set $\eta=5\cdot 10^{-4}$ and perform the computation described in section \ref{sec:ortho}. The result  we get is shown in Figure \ref{fig:sigma_visco}. It is a scaled version of the true conductivity $\sigma_\star$, which means that the contrast is recovered. So assuming we know the conductivity in a small region of $\Omega$  (or near the boundary $\partial \Omega$) we can re-scale the result, as shown in Figure \ref{fig:sigma_visco_scaled}.
When $\eta$ goes to zero, the solution of  (\ref{eq:Neumannregularized}) converges to the true potential $V_\star$ up to a scaling factor which goes to infinity. When $\eta$ is large, the scaling factor goes to one but the solution $U^{(\eta)}$ becomes a "smoothed out" version of $V_\star$.
This method allows an accurate reconstruction of the conductivity by solving only one partial differential equation. It covers the contrast accurately, provided we have a little bit of a prior information on $\sigma_\star$.

\begin{figure}
\begin{center}
\includegraphics[scale=0.3]{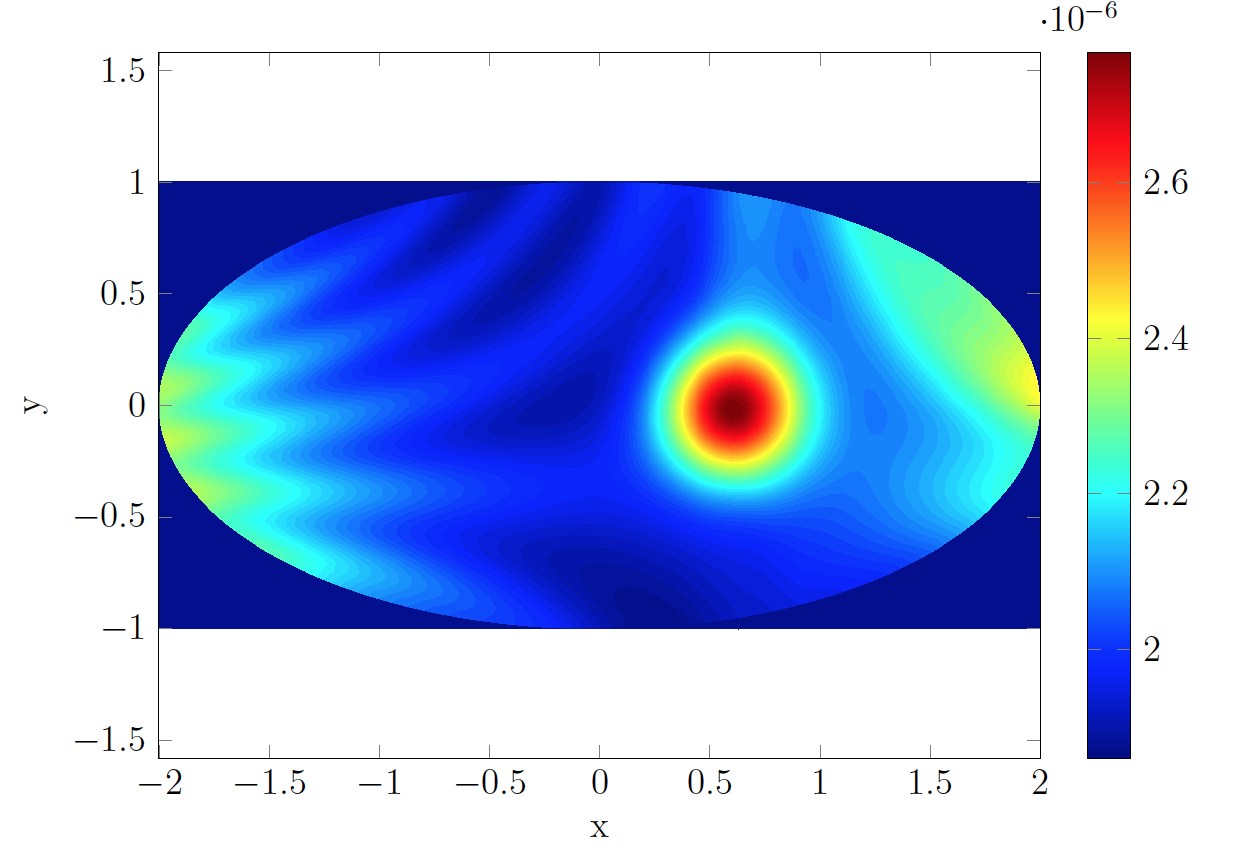}
\caption{\label{fig:sigma_visco} Conductivity recovered by the orthogonal field method before scaling.}
\end{center}
\end{figure}

\begin{figure}
\begin{center}
\includegraphics[scale=0.3]{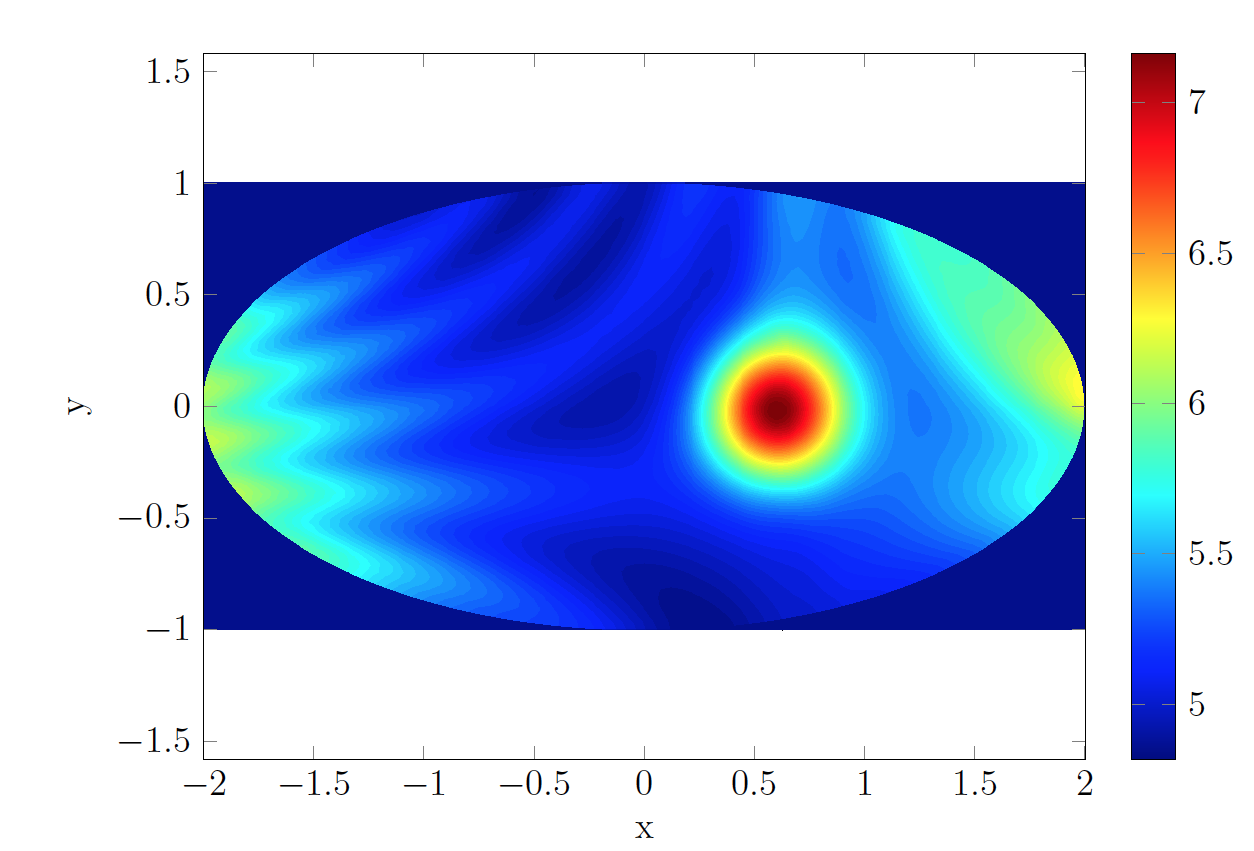}
\caption{\label{fig:sigma_visco_scaled} Conductivity recovered by the orthogonal field method after scaling.}
\end{center}
\end{figure}
Finally, we study the numerical stability with respect to measurement noise of the orthogonal field method. We compute the relative error defined by
\begin{equation*}
e:=\frac{\Vert \sigma_{\eta} - \sigma_\star \Vert_{L^2}}{\Vert \sigma_\star \Vert_{L^2}},
\end{equation*} averaged over $150$ different realizations of measurement noise on $\Jbf$. The results are shown in Figure \ref{fig:stabilityvisco}. We show the results of a reconstruction with noise level of $2\%$  (resp. $10\%$)  in Figure \ref{fig:visco02} (resp. Figure \ref{fig:visco10}). Clearly, the orthogonal method is quite robust with respect to measurement noise. 
\begin{figure}
\begin{center}
\includegraphics[scale=0.3]{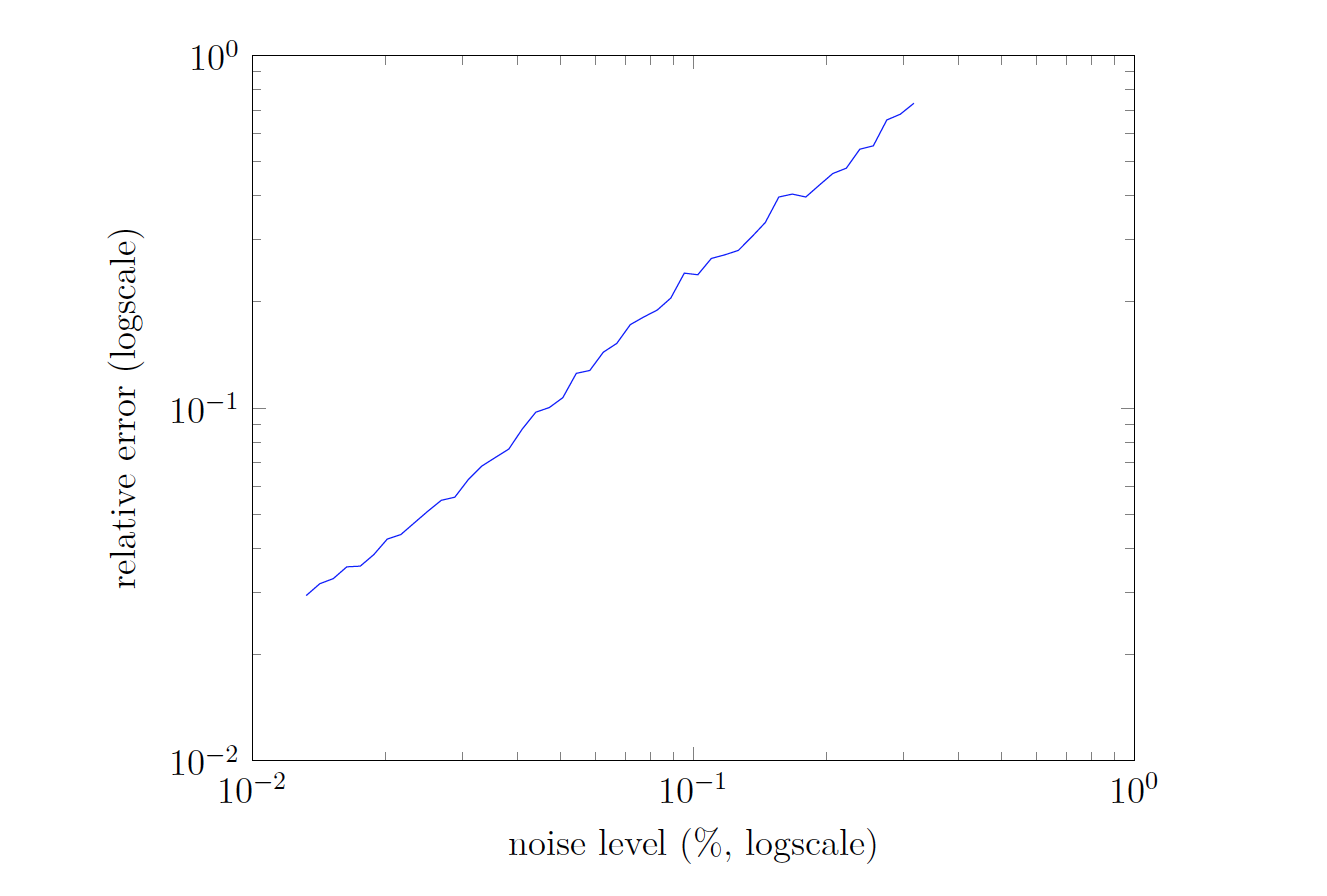}
\caption{\label{fig:stabilityvisco} Relative error with respect to measurement noise.}
\end{center}
\end{figure}

\begin{figure}
\begin{center}
\includegraphics[scale=0.3]{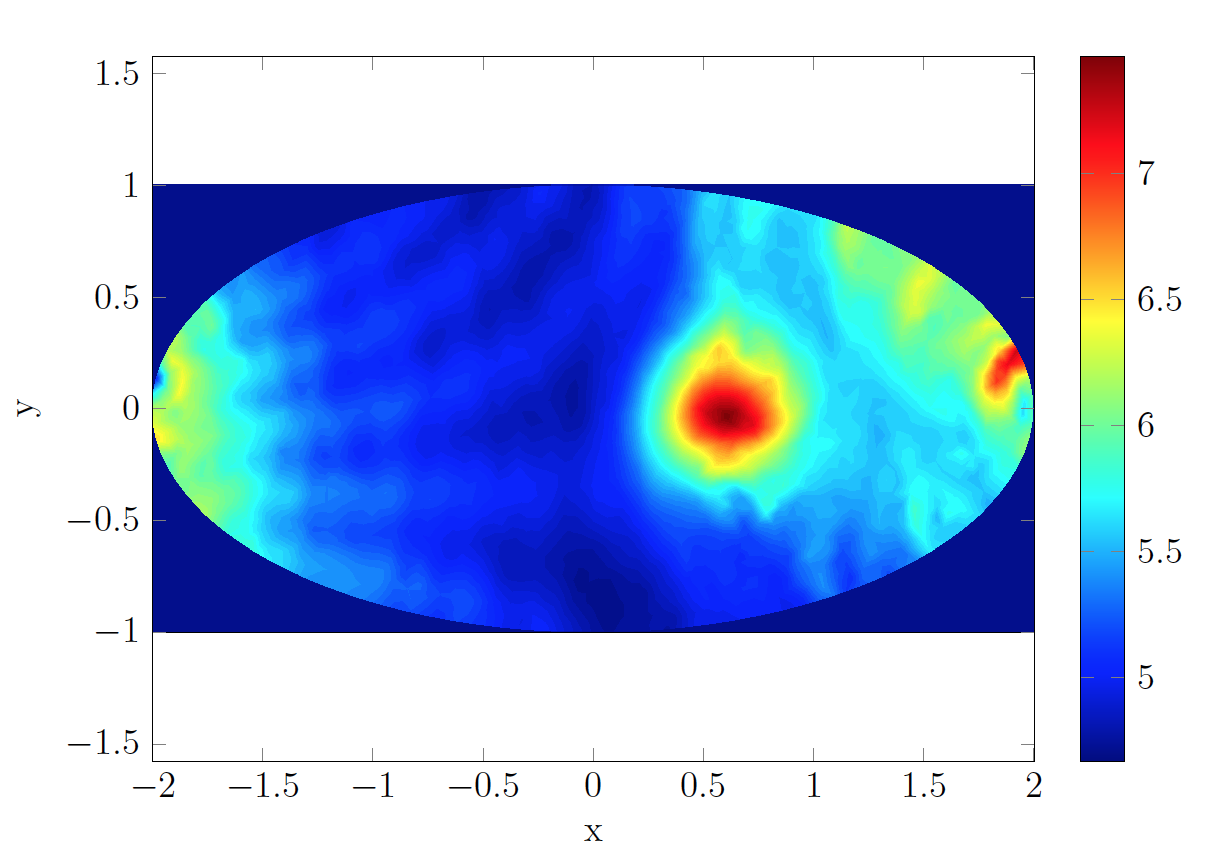}
\caption{\label{fig:visco02} Reconstruction with the orthogonal field method with 
measurement noise level of $2\%$.}
\end{center}
\end{figure}

\begin{figure}
\begin{center}
\includegraphics[scale=0.3]{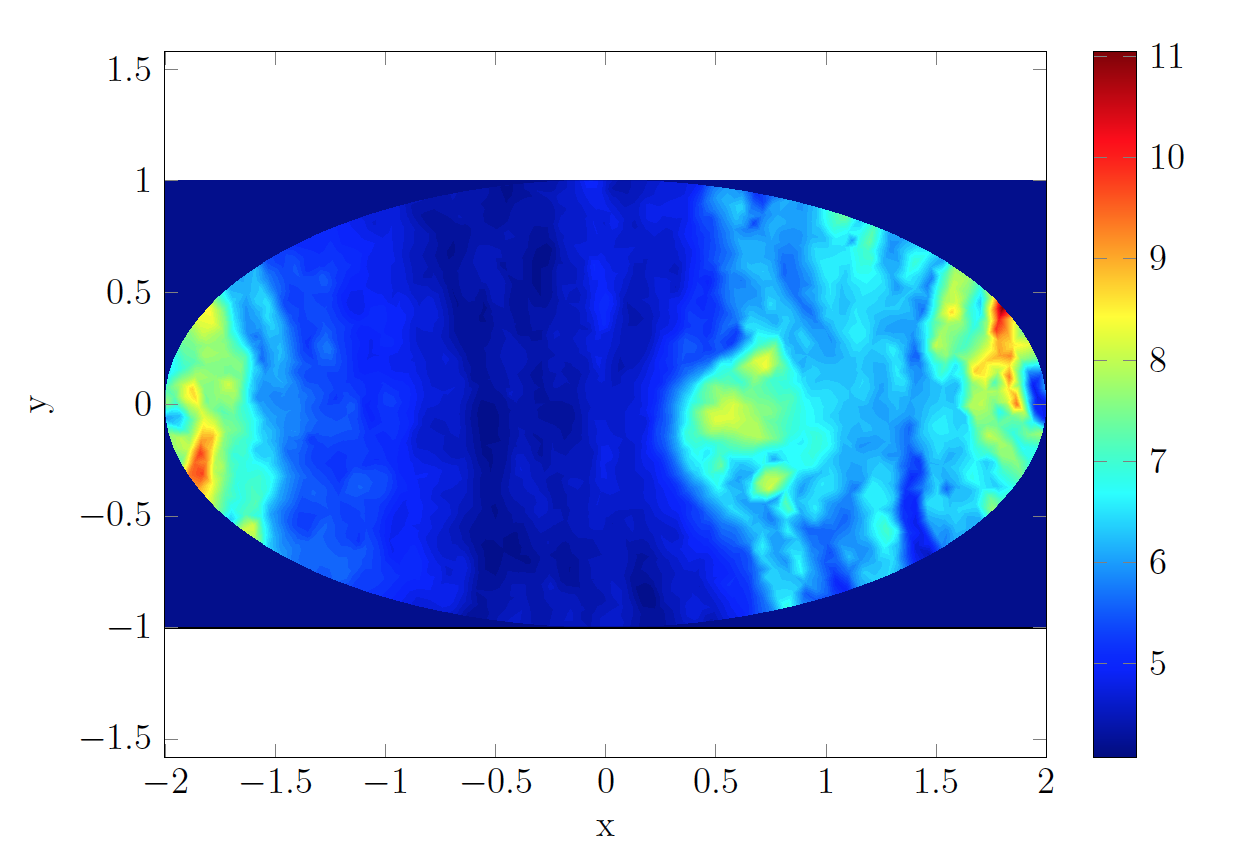}
\caption{\label{fig:visco10} Reconstruction with the orthogonal field method with
measurement noise level of  $10\%$.}
\end{center}
\end{figure}

\section{Concluding remarks}
In this paper we have presented a new mathematical and numerical framework for conductivity imaging using magnetoacoustic tomography with magnetic induction. We developed three different algorithms for conductivity imaging from boundary measurements of the Lorentz force induced tissue vibration. We proved convergence and stability properties of the three algorithms and compared their performance.  
The orthogonal field method performs much better than the optimization scheme and the fixed-point method in terms of both computational time and accuracy. Indeed, it is robust with respect to measurement noise.   In a forthcoming work, we intend to generalize our approach for imaging anisotropic conductivities by magnetoacoustic tomography with magnetic induction.

\end{document}